\documentclass[11pt]{amsart}
\usepackage{amsmath}
\usepackage{amssymb}
\usepackage{amsthm}
\usepackage{hyperref}
\usepackage[noabbrev,capitalize]{cleveref}
\usepackage{mathrsfs}
\usepackage{mathtools}
\usepackage{slashed}
\usepackage{tikz-cd}
\usepackage[shortlabels]{enumitem}

\setenumerate{label=(\roman*)}

\DeclareMathOperator{\Ric}{Ric}

\newcommand{\lp}{\langle}
\newcommand{\rp}{\rangle}

\newcommand{\mA}{\mathcal{A}}

\newcommand{\mH}{\mathcal{H}}

\def\sideremark#1{\ifvmode\leavevmode\fi\vadjust{\vbox to0pt{\vss
 \hbox to 0pt{\hskip\hsize\hskip1em
 \vbox{\hsize3cm\tiny\raggedright\pretolerance10000
 \noindent #1\hfill}\hss}\vbox to8pt{\vfil}\vss}}}

\newtheorem{theorem}{Theorem}[section]
\newtheorem{proposition}[theorem]{Proposition}
\newtheorem{lemma}[theorem]{Lemma}

\theoremstyle{definition}
\newtheorem{definition}[theorem]{Definition}

\theoremstyle{remark}

\numberwithin{equation}{section}

\makeatletter
\@namedef{subjclassname@2020}{\textup{2020} Mathematics Subject Classification}
\makeatother

\begin{document}

\title[Homological $k$-systole in $n$-manifolds]{Homological $k$-systole in $n$-manifolds with
  positive intermediate curvature}

\author{Jingche Chen}
\address{Department of Mathematical Sciences, Tsinghua University, 100084, Beijing, China}
\email{cjc23@mails.tsinghua.edu.cn}

\author{Han Hong}
\address{Department of Mathematics and statistics \\ Beijing Jiaotong University \\ Beijing \\
  China, 100044}
\email{hanhong@bjtu.edu.cn}

\begin{abstract}
In this paper, we prove optimal $k$-systolic inequalities and characterize the case of equality on
  closed $n$-dimensional Riemannian manifolds with positive intermediate curvature for $3\leq n\leq
  7$.  This unifies prior works of Bray-Brendle-Neves \cite{BrayBrenleNevesrigidity} and
  Chu-Lee-Zhu \cite{chuleezhu_n_systole}, and extends them to higher codimensions. The proof is
  inspired by our recent work on splitting theorems under intermediate curvature
  \cite{chenhong2026}. 
\end{abstract}

\maketitle

\section{Introduction}
This paper investigates the relationship between curvature positivity of an ambient manifold and
  the area of its minimal submanifolds. In dimension two, Toponogov's celebrated theorem
  \cite{toponogov_closed_geodesic} says the length of a simple closed geodesic in a two-sphere with
  positive Gauss curvature can be bounded from above. In higher dimensions, however, one cannot generally bound the area of an
  embedded minimal hypersurface in a Riemannian manifold of positive curvature. For example, the
  area of a closed embedded minimal hypersurface in $\mathbb{S}^{3}$ can be arbitrarily large
  (\cite{zhou_multiplicity_one_conj}).  

If we restrict attention to \textit{area-minimizing} (rather than merely minimal) objects, more
  positive results emerge.  Define
\begin{equation*}
    \mathcal{A}_{\mathbb{S}^2}(M, g) := \inf \left\{ \operatorname{Area}_g(\Sigma) | \Sigma  \text{
      is a non-contractible $2$-sphere in } M \right\}.
\end{equation*}
Motivated by Schoen-Yau's work on incompressible minimal surfaces
  \cite{Schoen-Yau-incompressible-minimal-surfaces} and the existence of smooth area minimizers in 
  nontrivial homotopy classes in $\pi_2(M)$ \cite{meeks-yau-existence}, Bray-Brendle-Neves
  \cite{BrayBrenleNevesrigidity} studied upper bounds of $\mathcal{A}_{\mathbb{S}^2}(M^3, g)$ for
  closed, oriented $3$-dimensional manifolds $M^3$ with positive scalar curvature. They proved
\[\mathcal{A}_{\mathbb{S}^2}(M, g)\cdot \inf R_g\leq 8\pi,\]
with equality if and only if $M$ is isometrically covered by $\mathbb{S}^2\times \mathbb{R}.$ The
  inequality (usually referred to as the spherical systolic inequality) follows from the stability
  inequality of the minimizer, while the rigidity part is more intriguing and is proved by using a
  local foliation argument.  A similar inequality and its rigidity for embedded
  $\mathbb{RP}^{2}$ were later established by Bray-Brendle-Eichmair-Neves
  \cite{bray-brendle-eichmair-neves}. 
Subsequently, Stern \cite{stern-scalar-curvature-harmonic-maps} employed a harmonic map level-set
  approach and obtained the following homological systolic inequality
\[\operatorname{sys}_2(M,g)\cdot \inf R_g\leq 8\pi,\]
again with equality if and only if $M$ is isometrically covered by $\mathbb{S}^2\times \mathbb{R}$,
  where $\operatorname{sys}_2(M,g)$ is the least area of nonseparating surfaces in $M$, i.e.,
\[\operatorname{sys}_2(M,g)=\inf\{\operatorname{Area}_g(\Sigma)|\Sigma\ \text{is a smooth surface},
  [\Sigma]\neq 0\in H_2(M)\}.\] 

There are several generalizations to higher dimensions. First, if one considers $n$-dimensional
  Riemannian manifolds with nontrivial second homotopy group $\pi_2(M)$, the spherical systolic
  inequality fails due to the example $\mathbb{S}^2(r_1)\times \mathbb{S}^{n-2}(r_2)$ for $n\geq 4$.
   Thus,
  Zhu \cite{zhu-spherical-systole-dimension-two} imposed a topological condition  requiring $M^n$
  to admit a nonzero-degree map to $\mathbb{S}^2\times T^{n-2}$ for $n\leq 7$, and proved
\[\mathcal{A}_{\mathbb{S}^2}(M, g) \cdot \inf R_g\leq 8\pi,\]
with equality if and only if $M$ is isometrically covered by $\mathbb{S}^2\times \mathbb{R}^{n-2}.$
  A hyperspherical systolic inequality remains unknown due to the lack of existence theory for area
  minimizers in higher homotopy classes. On the other hand, a higher dimensional homological
  systolic inequality was recently studied by Chu-Lee-Zhu \cite{chuleezhu_n_systole}. They showed
  that if an $n$-dimensional closed Riemannian manifold $M$ has nontrivial $(n-1)$-homology group
  and its biRic curvature is bounded below by $n-2$, then
\[\operatorname{sys}_{n-1}(M,g)\leq |\mathbb{S}^{n-1}|,\]
with equality if and only if $M$ is isometrically covered by $\mathbb{S}^{n-1}\times \mathbb{R}$.
  Here 
\begin{equation*}
    \operatorname{sys}_{n-1}(M, g) := \inf \left\{ \operatorname{Area}_g(\Sigma) \;\middle|\;
\begin{aligned}
\Sigma  &\text{ is a smooth hypersurface in } M, \\
&[\Sigma] \neq 0 \in H_{n-1}(M)
\end{aligned} \right\}.
\end{equation*}
This result holds in all dimensions under a generic regularity hypothesis. A crucial fact
  worth mentioning is that under a positive biRic curvature assumption, a stable minimal
  hypersurface inherits a positive spectral Ricci lower bound. 
  This allows one to estimate the area of the minimizer via Antonelli-Xu's comparison theorem
  \cite{antonelli-xu}, which is a generalization of the classical Bishop-Gromov volume comparison
  theorem. 

We first recall the notion of the $C_m$-curvatures introduced by Brendle-Hirsch-Johne
  \cite{brendlegeroch'sconjecture}. 

\begin{definition}

    Suppose $(M,g)$ is an $n$-dimensional manifold. For given orthonormal vectors
      $\{e_{1},...,e_{m}\}$ in $T_xM$ at the point $x\in M$, one can extend them to an orthonormal
      basis $\{e_{1},...,e_{n}\}$ of $T_{x}M$. 
    The $m$-intermediate curvature $C_{m}$ of the orthonormal vectors $\{e_{1},...,e_{m}\}$ is
      defined by 
    \[
    C^M_{m}(e_{1},...,e_{m}) = \sum_{p=1}^{m}\sum_{q=p+1}^{n}Rm(e_{p},e_{q},e_{p},e_{q}).\]
In particular, $C^M_1(e_1)=\Ric_M(e_1,e_1)$ and $C^M_{n-1}(e_1,\cdots,e_{n-1})=R_M/2.$
   Let 
\[C^M_m(x):=\min\{C^M_m(e_1,\cdots,e_m)| \{e_1,\cdots,e_m\}\ \text{are orthonormal in}\ T_xM\}.\]
\end{definition}

\begin{definition}
For an $n$-dimensional Riemannian manifold $(M,g)$ with nontrivial homology group $H_{k}(M;\mathbb{Z})$, we
  define the $k$-systole as
\begin{equation*}
    \text{sys}_{k}(M, g) := \inf \left\{ \mathcal{H}^k(\Sigma) \;\middle|\;
\begin{aligned}
\Sigma  &\text{ is an oriented smooth submanifold in } M, \\
&[\Sigma] \neq 0 \in H_{k}(M)
\end{aligned} \right\}.
\end{equation*}
When $k=n$, $\text{sys}_{k}(M, g)$ denotes the volume of $M$.
\end{definition}

Our main result is the following.
    \begin{theorem}\label{vol inequality in Cm case in the introduction}
    Assume \(3 \leq  n \leq  7\), \(1 \leq  m \leq  n-1\). Let $(N^{n},g)$ be a closed, orientable
      $n$-dimensional Riemannian manifold with positive $m$-intermediate curvature $C_{m}\geq n-m$.
      Suppose there exists a smooth map $f:N\rightarrow E^{n-m+1}\times \mathbb{T}^{m-1}$
      with nonzero degree, where $E$ is an oriented $(n-m+1)$-dimensional closed manifold.  Then
    \begin{equation}\label{maininequality}
        \operatorname{sys}_{n-m+1}(N,g)\leq  |\mathbb{S}^{n-m+1}|.
    \end{equation}
\end{theorem}

This theorem recovers several known systolic inequalities. When $m=1$, it is just the
  Bishop-Gromov volume comparison theorem. When $n=3$ and $m=2$, it recovers the results of Stern
  \cite{stern-scalar-curvature-harmonic-maps} and Bray-Brendle-Neves
  \cite{BrayBrenleNevesrigidity}. When $4\leq n\leq7$ and $m=n-1$, it
  recovers the result of Zhu \cite{zhu-spherical-systole-dimension-two}. When $4\leq n\leq7$
  and $m=2$, it recovers the result of Chu-Lee-Zhu
  \cite{chuleezhu_n_systole}. Thus the new cases covered by Theorem
  \ref{vol inequality in Cm case in the introduction} occur when $5\leq n\leq7$ and
  $3\leq m\leq n-2$. The proof also gives a unified approach to these previously known results.

The topological assumption is necessary, otherwise, as noted earlier, the product
  $\mathbb{S}^{m-1}(r_1)\times \mathbb{S}^{n-m+1}(r_2)$ $(m\geq 3)$ gives a counterexample. Indeed, by
  choosing $r_1$ sufficiently small, one can arrange $C_m\geq n-m$ independently of how large
  $r_2$ is. On the other hand, the nontrivial class
  $[\{p\}\times \mathbb{S}^{n-m+1}(r_2)]\in
  H_{n-m+1}\big(\mathbb{S}^{m-1}\times\mathbb{S}^{n-m+1}\big)$ has least volume equal to
  $r_2^{\,n-m+1}|\mathbb{S}^{n-m+1}|$, which can be made arbitrarily large. Thus the desired
  systolic bound fails without the topological assumption. We remark that the dimension
  restriction is not only due to regularity issues. Even though a generic regularity
  hypothesis is assumed, we cannot prove such an inequality in higher dimensions for all $m$.

The following theorem gives the rigidity statement.
\begin{theorem}\label{vol inequality in Cm case rigidity}
   Under the assumptions of Theorem \ref{vol inequality in Cm case in the introduction}, assume either
     \(3 \leq  n \leq  6\), \(1 \leq  m \leq  n-1\), or $n=7$, $m\in\{1,2,5,6\}$.
    Then equality holds in \eqref{maininequality} if and only if $N$ is isometrically covered by
      $\mathbb{S}^{n-m+1}\times \mathbb{R}^{m-1}$.
\end{theorem}

We outline the proof of Theorem \ref{vol inequality in Cm case in the introduction}, which is
  inspired by our recent work on splitting theorems under intermediate curvature. The
  topological assumption gives nontrivial homology classes by successively capping with the pulled
  back torus classes (see Lemma \ref{topology inherit}). In each such class we choose a minimizer of
  the corresponding weighted functional; equivalently, this is an area minimizer for a conformally
  changed metric. Since $n\leq7$, these minimizers are smooth, and hence we obtain a stable
  weighted $(m-1)$-slicing
  $\Sigma_{m-1}\subset \cdots\subset \Sigma_{0}=N$. This slicing is different from that of
  Brendle-Hirsch-Johne.
  Hence we only need to show that $|\Sigma_{m-1}|\leq
  |\mathbb{S}^{n-m+1}|$, since $\Sigma_{m-1}$ represents a nontrivial class in $H_{n-m+1}(N)$.
   An important feature of $\Sigma_{m-1}$ is that $\Sigma_{m-1}$ satisfies the
  uniformly positive spectral Ricci curvature condition (see Theorem \ref{spectral inheritance}).
  Thus, the area estimate follows by Antonelli-Xu's volume comparison theorem.

We also outline the proof of Theorem \ref{vol inequality in Cm case rigidity}. The case $m=1$ is
  exactly the equality case in the Bishop-Gromov volume comparison theorem, so assume
  $2\leq m\leq n-1$. Under the equality assumption, we use the same weighted slicing
  $\Sigma_{m-1}\subset\cdots\subset\Sigma_1\subset\Sigma_0=N$. Equality in the final volume
  estimate forces the terminal slice $\Sigma_{m-1}$ to be round and forces the final spectral
  function to be constant; this is the initial step of the rigidity induction.

The rest of the proof proceeds by upward induction in dimension, or equivalently by reverse
  induction on the index $k$. Suppose that the desired splitting has already been proved for
  $\Sigma_{k+1}$. The hypersurface $\Sigma_{k+1}\subset\Sigma_k$ is chosen as a minimizer of the
  weighted functional determined by the spectral function $u_k$. The spectral inheritance theorem
  gives the inherited curvature condition on $\Sigma_{k+1}$, while its rigidity statement gives
  total geodesicity, a Hessian--Ricci identity, and the vanishing of the normal derivative of
  $u_k$ along $\Sigma_{k+1}$.

The main remaining point is to promote these identities from one hypersurface to the whole ambient
  slice $\Sigma_k$. For this we use the surface-capture, or metric-deformation, technique. It was
  first introduced by Liu \cite{Liu-nonnegative-Ricci-curvature} for noncompact manifolds and later
  extended by Chu, Lee, and Zhu \cite{chuleezhu_n_systole} to the compact setting
 (see also \cite{Carlotto-Chodosh-Eichmair-PMT,chodosh-eichmair-moraru-splitting,chu-lee-zhu-kaler-splitting,hong-wang-spectral-splitting,hong-wang-splitting-scalarcurvature}).
  After a local conformal deformation near an arbitrary point, the rigidity of the deformed
  minimizers gives a sequence of totally geodesic hypersurfaces converging from one side. An
  open-closed argument then shows that through every point of $\Sigma_k$ there is a weighted
  minimizing hypersurface in the same homology class. Applying the rigidity identities along all
  these hypersurfaces implies that $u_k$ is constant and that $\Ric_{\Sigma_k}\geq0$. The standard
  splitting lemma then shows that $\Sigma_k$ is isometrically covered by
  $\mathbb{S}^{n-m+1}\times\mathbb{R}^{m-k-1}$. Repeating this argument reaches $k=1$; the final
  step $k=0$ is the unweighted case and gives the desired covering of $N$.

One advantage of the present approach is its uniform treatment of intermediate curvature
  conditions. In earlier work, the scalar curvature case relied on toric symmetrization, while the
  biRicci case was treated using the weighted slicing of Brendle--Hirsch--Johne. Here the spectral
  inheritance principle for intermediate curvature, an observation in \cite{chenhong2026}, replaces
  these separate mechanisms and applies to the full range of intermediate curvature conditions
  considered in this paper.

  \vskip.2cm
The cases $n=7$, $m=3,4$ are excluded from the rigidity statement for a technical reason in the
  induction. In the surface-capture step one needs a strict numerical inequality
  $D(n-k,m-k)>\frac{k-1}{2k}$ in order to use the rigidity part of the spectral inheritance theorem.
  This strict inequality holds in the range stated above but fails exactly in the borderline cases
  $n=7$, $m=3,4$. Thus our argument does not force the identities needed to propagate the splitting
  in these two cases. It would be interesting to know whether the equality case of
  \eqref{maininequality} is still rigid when $(n,m)=(7,3)$ or $(7,4)$, namely, whether equality
  still implies that $N$ is isometrically covered by
  $\mathbb{S}^{n-m+1}\times\mathbb{R}^{m-1}$, or whether there exists a counterexample to this
  rigidity statement in one of these borderline cases.

\vskip.2cm
The paper is structured as follows. In Section \ref{section2}, we recall the spectral curvature
  inheritance and topology inheritance and discuss the standard splitting under the assumption of
  nonnegative Ricci curvature.  In Section
  \ref{section3}, we prove the inequality part of the main theorem. In Section \ref{section4}, we
  complete the characterization of equality in the main theorem.

\subsection{Acknowledgements}
 The first author would like to thank his supervisor Prof Haizhong Li for his encouragement and
   support. The second author is supported by the Fundamental Research Funds for the Central Universities No. YA26JBMC00040, No. 2024XKRC008 and NSFC No. 12401058.

\section{Preliminary}\label{section2}

\begin{definition}
We say that $N$ has 
spectral $(k,m,\lambda)$-intermediate curvature if there exist a positive function $u \in
  C^{2,\beta}(N)$ for some $\beta \in (0,1)$ and constants $k \geq 0, \lambda$ such that
\[
 -k \Delta_N u + C^N_m\, u \geq \lambda u. 
\]
In particular, when $k = 0$, this reduces to $C_{m}^{N}\geq \lambda$.
\end{definition}
When $m=1$, the definition reduces to the spectral $(k,\lambda)$-Ricci curvature.
\begin{definition}
    We say that $N$ has spectral $(k,\lambda)$-Ricci curvature if there exist a positive function  
      $u \in C^{2,\beta}(N)$ for some $\beta \in (0,1)$ and constants $k \geq 0,\lambda$ such that
    \[-k \Delta_N u+\Ric_N u \geq \lambda u.\]
    If $\lambda>0$, we say that $N$ has positive spectral $(k,\lambda)$-Ricci curvature.
\end{definition}

For spectral Ricci curvature, we have the following important result.
\begin{theorem}[\cite{antonelli-xu}]\label{volume bound under spectral Ric}
    Let $N^n$ be an $n$-dimensional closed smooth Riemannian manifold with $n\geq 3$, and let
      $0\leq \gamma \leq \frac{n-1}{n-2}$. Assume that $N$ has positive spectral
      $(\gamma,\lambda)$-Ricci curvature.
    Then 
    \[\operatorname{vol}(N)\leq \left(\frac{n-1}{\lambda}\right)^{n/2}|\mathbb{S}^{n}|.\]
    Moreover, if equality holds, then the function $u$ in the definition of spectral Ricci
      curvature is constant and $N$ is isometric to the round sphere of sectional curvature
      $\frac{\lambda}{n-1}$.
\end{theorem}

For integers $n\geq 3 $ and $1\leq m\leq n-1$ such that $m^2-mn+2n-2>0$ and $m^2-mn+m+n\geq0$, we
  denote 
\[D(n,m)=\min\left\{\frac{m}{2m-2},\frac{1}{n-m},\frac{m^2-mn+m+n}{2(m^2-mn+2n-2)}\right\}.\]
We prove the following theorem.

\begin{theorem}\label{spectral inheritance}
    Let $N$ be an $n$-dimensional closed Riemannian manifold with positive spectral
      $(k,m,\lambda)$-intermediate curvature for $0\leq k<4$ \textnormal{(i.e. $\lambda>0$)}. Suppose that $\Sigma$ is
      a closed $(n-1)$-dimensional hypersurface that is a minimizer of $\mA_{k}=\int
      u^k d\mathcal{H}^{n-1}$ in a nontrivial homology class. If $k=0$, or if
      $0<k<4$ and $D(n,m)\geq \frac{k-1}{k}$, then $\Sigma$ admits positive spectral
      $(\frac{4}{4-k},m-1,\lambda)$-intermediate curvature.

    Moreover, the following rigidity statements hold:
    \begin{enumerate}[label=(\alph*)]
        \item If $0<k<4$, $C_{m-1}^{\Sigma}=\lambda$, and
          $D(n,m)> \frac{k-1}{k}$, then $|A^{\Sigma}|=0$,
          $\nabla_Nu=0$, and
          \[
              k(\nabla_N^2\log u)(\nu,\nu)=\Ric_N(\nu,\nu)
          \]
          on $\Sigma$.
        \item If $k=0$ and $C_{m-1}^{\Sigma}=\lambda$, then $|A^{\Sigma}|=0$ and
          $\Ric_N(\nu,\nu)=0$ on $\Sigma$.
    \end{enumerate}
\end{theorem}

\begin{proof}
    Assume first that $0<k<4$.
    Since $\Sigma$ is the minimizer of $\mA_{k}$, we calculate the first and second variations of
      $\mA_{k}$. The first variation is
  
       \begin{equation*}
        \frac{d}{d t}\bigg|_{t=0}\mA_{k}(\Sigma_t)=\int_{\Sigma}\left(H_\Sigma
          u^k+\lp\nabla^{N}u^k,\nu\rp\right)\psi d\mH^{n-1}
    \end{equation*}
    Thus $H_\Sigma=-k \nabla^N_\nu\log u$ on $\Sigma$. 
    The second variation is:
    \begin{equation*}
        \begin{split}
            \frac{d^2}{d t^2}\bigg|_{t=0}\mA_{k}(\Sigma_t)&=\int_{\Sigma}|\nabla^{\Sigma}\psi|^2
              u^k-\left(|A_{\Sigma}|^2+\Ric_N(\nu,\nu)\right)\psi^2 u^k\\
            &+\int_\Sigma k(k-1)u^{k-2}\left(\nabla^N_{\nu}u\right)^2\psi^2
              +\int_{\Sigma}ku^{k-1}\left(\Delta_N u-\Delta_{\Sigma} u\right)\psi^2\\
            &\geq 0.
        \end{split}
    \end{equation*}
    Let $\psi=u^{-k/2}\phi$. 
    Then
    \[|\nabla^{\Sigma}\psi|^2
      =u^{-k}|\nabla^\Sigma\phi|^2+\frac{k^2}{4}\phi^2u^{-k-2}|\nabla^\Sigma
      u|^2-ku^{-k-1}\phi\nabla^\Sigma\phi\cdot\nabla^\Sigma u.\]
    We obtain
    \begin{equation}\label{A_k-stability inequality}
    \begin{split}
        0&\leq \int_\Sigma |\nabla^\Sigma \phi|^2+(ku^{-1}\Delta_N
          u-|A_\Sigma|^2-\Ric_N(\nu,\nu)+\frac{k-1}{k}H_{\Sigma}^2)\phi^2\\
        & +\int_\Sigma (\frac{k^2}{4}-k)\phi^2|\nabla^\Sigma \log
          u|^2+k\phi\nabla^\Sigma\phi\cdot\nabla^\Sigma\log u
    \end{split}
    \end{equation}
    where we have used the critical equation. Using the Cauchy-Schwarz inequality
    \begin{equation}\label{cauchy ineq}
    k|\phi \nabla^\Sigma\phi||\nabla^\Sigma\log u|\leq \epsilon \phi^2|\nabla^\Sigma \log
      u|^2+\frac{k^2}{4\epsilon}|\nabla^\Sigma \phi|^2
    \end{equation}
    and taking $\epsilon=k-k^2/4$, we have
    \[0\leq \int_\Sigma \frac{4}{4-k}|\nabla^\Sigma\phi|^2+(ku^{-1}\Delta_N
      u-|A_\Sigma|^2-\Ric_N(\nu,\nu)+\frac{k-1}{k}H_{\Sigma}^2)\phi^2.\]
    
    It follows from the assumption that $N$ has positive spectral $(k,m,\lambda)$-intermediate
      curvature that
    \begin{equation*}
        0\leq\int_\Sigma
          \frac{4}{4-k}|\nabla^\Sigma\phi|^2+(C_m^N-\lambda-|A_\Sigma|^2-\Ric_N(\nu,\nu)+\frac{k-1}{k}H_{\Sigma}^2)\phi^2.
    \end{equation*}
    Let $\nu$ be the unit normal vector field of $\Sigma$ in $N$. At a point $p\in \Sigma$, we
      assume that $\{e_2,\cdots,e_n\}$ is an  orthonormal basis of $T_p\Sigma$ such that
      $\{e_1=\nu,e_2,\cdots,e_n\}$ is an orthonormal basis of $T_pN$ and
      $C_{m-1}^\Sigma=C_{m-1}^\Sigma(e_2,\cdots,e_m).$ Then the Gauss equation yields
\begin{align*}
    C_m^N&\leq \sum_{i=2}^{m}\sum_{j=i+1}^n R^N_{ijij}+\Ric_N(\nu,\nu)\\
    &=\sum_{i=2}^{m}\sum_{j=i+1}^n (R^\Sigma_{ijij}-h_{ii}h_{jj}+h_{ij}^2)+\Ric_N(\nu,\nu)\\
      &=C^\Sigma_{m-1}-\sum_{i=2}^{m}\sum_{j=i+1}^n (h_{ii}h_{jj}-h_{ij}^2)+\Ric_N(\nu,\nu)
\end{align*}   
where $h(\cdot,\cdot)$ is the second fundamental form of $\Sigma$ in $N$ with respect to the unit
  normal $\nu$. Now we have
\begin{equation}\label{inequality_1_lemma}
0\leq \int_\Sigma
  \frac{4}{4-k}|\nabla^\Sigma\phi|^2+\left(C_{m-1}^\Sigma-\lambda-|A_\Sigma|^2-\sum_{i=2}^{m}\sum_{j=i+1}^n (h_{ii}h_{jj}-h_{ij}^2)+\frac{k-1}{k}H_{\Sigma}^2\right)\phi^2.
\end{equation}

In the following we directly cite a sharp algebraic inequality proved by Chen \cite[Lemma
  5.8]{chenshuli_end}:
\[|A_\Sigma|^2+\sum_{i=2}^{m}\sum_{j=i+1}^n (h_{ii}h_{jj}-h_{ij}^2)\geq D(n,m)H_\Sigma^2\]
if $m^2-mn+2n-2>0$ and $m^2-mn+m+n\geq 0$.

Using the above inequality, we have
\[0\leq\int_\Sigma \frac{4}{4-k}|\nabla^\Sigma\phi|^2+(C_{m-1}^\Sigma-\lambda)\phi^2.\]
When $\Sigma$ is compact, this implies that there exists a positive function $v$ on $\Sigma$ such that
\[-\frac{4}{4-k}\Delta_\Sigma v+(C^\Sigma_{m-1}-\lambda)v=\lambda_1 v\geq 0.\]

If $C_{m-1}^{\Sigma}=\lambda$, then taking $\phi=1$ in \eqref{inequality_1_lemma}, we obtain
\[
0\leq \int_\Sigma -|A_\Sigma|^2-\sum_{i=2}^{m}\sum_{j=i+1}^n
  (h_{ii}h_{jj}-h_{ij}^2)+\frac{k-1}{k}H_{\Sigma}^2.
\]
Since $D(n,m)>\frac{k-1}{k}$, it follows that $\Sigma$ is totally geodesic.
Moreover, equality in the Cauchy-Schwarz step in \eqref{cauchy ineq} forces
  $\nabla^{\Sigma}\log u=0$. Combining this with
  $H_{\Sigma}=-k\nabla_{\nu}^{N}\log u$, we get $\nabla_{N}u=0$ along $\Sigma$.
It remains to prove the last identity.  Along $\Sigma$ we have equality in the spectral curvature condition,

\[
    -k\Delta_N u+C_m^Nu=\lambda u,
\]
and equality in the Gauss equation,
\[
    C_m^N=C_{m-1}^{\Sigma}+\Ric_N(\nu,\nu).
\]
Since $C_{m-1}^{\Sigma}=\lambda$, the second equality gives
\[
    C_m^N-\lambda=\Ric_N(\nu,\nu).
\]
On the other hand, $\nabla_Nu=0$ along $\Sigma$ implies
  $u|_{\Sigma}$ is constant and $\nabla_\nu^N u=0$. Hence
\[
    \Delta_Nu=\Delta_\Sigma u+\nabla_N^2u(\nu,\nu)+H_\Sigma\nabla_\nu^N u
    =\nabla_N^2u(\nu,\nu)
    =u(\nabla_N^2\log u)(\nu,\nu)
\]
on $\Sigma$. Substituting this into
  $-k\Delta_Nu+C_m^Nu=\lambda u$ yields
\[
    k(\nabla_N^2\log u)(\nu,\nu)=C_m^N-\lambda.
\]
Thus
\[
    k(\nabla_N^2\log u)(\nu,\nu)=\Ric_N(\nu,\nu)
\]
on $\Sigma$.

It remains to treat the case $k=0$. In this case $\mA_0$ is the area functional and the spectral
  curvature assumption reduces to $C_m^N\geq \lambda$. Since $\Sigma$ is area-minimizing, it is
  stable minimal, and hence
\[
    0\leq \int_\Sigma |\nabla^\Sigma\phi|^2
      -\left(|A_\Sigma|^2+\Ric_N(\nu,\nu)\right)\phi^2
\]
for every smooth function $\phi$ on $\Sigma$. The same Gauss equation as above gives
\[
    C_m^N\leq C_{m-1}^{\Sigma}
    -\sum_{i=2}^{m}\sum_{j=i+1}^{n}(h_{ii}h_{jj}-h_{ij}^2)+\Ric_N(\nu,\nu),
\]
and Chen's algebraic inequality gives
\[
    |A_\Sigma|^2+\sum_{i=2}^{m}\sum_{j=i+1}^{n}(h_{ii}h_{jj}-h_{ij}^2)\geq 0,
\]
because $H_\Sigma=0$. Combining these inequalities with $C_m^N\geq\lambda$, we obtain
\[
    0\leq \int_\Sigma |\nabla^\Sigma\phi|^2+(C_{m-1}^{\Sigma}-\lambda)\phi^2.
\]
Therefore there exists a positive function $v$ on $\Sigma$ satisfying
\[
    -\Delta_\Sigma v+(C_{m-1}^{\Sigma}-\lambda)v\geq 0.
\]
Hence $\Sigma$ admits positive spectral $(1,m-1,\lambda)$-intermediate curvature, which is the
  claimed conclusion when $k=0$.
If in addition $C_{m-1}^{\Sigma}=\lambda$, then taking $\phi=1$ in the stability inequality gives
\[
    0\leq -\int_\Sigma \left(|A_\Sigma|^2+\Ric_N(\nu,\nu)\right).
\]
On the other hand, the Gauss equation and $C_m^N\geq \lambda=C_{m-1}^\Sigma$ imply
\[
    \Ric_N(\nu,\nu)\geq
    \sum_{i=2}^{m}\sum_{j=i+1}^{n}(h_{ii}h_{jj}-h_{ij}^2).
\]
Hence
\[
    |A_\Sigma|^2+\Ric_N(\nu,\nu)\geq
    |A_\Sigma|^2+\sum_{i=2}^{m}\sum_{j=i+1}^{n}(h_{ii}h_{jj}-h_{ij}^2)\geq0.
\]
Thus equality holds throughout. By the equality case of the algebraic inequality, $A_\Sigma=0$.
The stability inequality above then gives $\Ric_N(\nu,\nu)=0$ on $\Sigma$.

\end{proof}

The following lemma provides the existence of an area-minimizing hypersurface whenever $N$ satisfies
  a topological condition. Moreover, the hypersurface also inherits the topological condition.

\begin{lemma}\label{topology inherit}
    Let $N$ be an $n$-dimensional closed orientable manifold. Suppose $3\leq n\leq 7$, $1\leq m\leq n-1$ and there exists a smooth map $f:N\rightarrow E\times
      \mathbb{T}^{m-1}$ with nonzero degree. 
      Let $0\leq k\leq n-2$ and let $W\subset N$ be a closed, oriented
      $(n-k)$-dimensional smooth submanifold representing the homology class
      $[N]\frown(f^{*}(d\theta_{1})\smile \cdots\smile f^{*}(d\theta_{k}))\in
      H_{n-k}(N)$. Then there exists an area-minimizing hypersurface
      $\Sigma\subset W$ such that $\Sigma$ represents the homology class
      $[N]\frown(f^{*}(d\theta_{1})\smile \cdots\smile f^{*}(d\theta_{k+1}))\in H_{n-k-1}(N)$.
\end{lemma}
\begin{proof}
    Let $H_{k+1}\subset N$ be a smooth closed oriented hypersurface representing
      $[N]\frown f^{*}(d\theta_{k+1})$. After a small perturbation, we may assume that
      $H_{k+1}$ intersects $W$ transversely. Then $H_{k+1}\cap W$ is a smooth closed oriented
      hypersurface in $W$. We use $P$ to denote the Poincare duality map. By the standard
      intersection formula for Poincare duals,
    \[
    \begin{aligned}
    &[H_{k+1}\cap W] = P(P^{-1}([H_{k+1}])\smile P^{-1}([W]))\\
    &= [N]\frown(f^{*}(d\theta_{1})\smile \cdots\smile f^{*}(d\theta_{k+1}))\neq 0\in H_{n-k-1}(N).
    \end{aligned}
    \]
    Then
    \[
    [H_{k+1}\cap W]\neq 0 \in H_{n-k-1}(W).
    \]
    Thus, there exists $\Sigma\subset W$ which is a minimizing hypersurface representing $[H_{k+1}\cap W]\in H_{n-k-1}(W)$. Moreover,
    \[
    [\Sigma]= [H_{k+1}\cap W]\in H_{n-k-1}(N).
    \]
\end{proof}
\subsection{Splitting}
We state the following splitting result which will be used several times in our main proof.
\begin{lemma}\label{splitting-under-ric}
    Suppose $M^{n}$ is a closed Riemannian manifold with non-negative Ricci curvature and
      there exists a closed, embedded, two-sided, locally area-minimizing hypersurface $\Sigma$ in
      $M$.
    Then there exists a neighborhood of $\Sigma$ in $M$ that is isometric to $(\Sigma\times
      (-\epsilon,\epsilon),g_{\Sigma}+dt^2)$. Moreover, if $\Sigma$ is area-minimizing in the
      homology class $[\Sigma]\neq 0\in H_{n-1}(M)$, then $M$ is isometrically covered by
      $(\Sigma\times \mathbb{R},g_{\Sigma}+dt^2)$.  
\end{lemma}

This follows from a special case of Theorem 3.2 (d) in \cite{heintze-karcher-submanifold}. Consider
  the foliation of local equidistant hypersurfaces around $\Sigma$. One can show that the area is
  non-increasing and the normal vector field of the foliation is parallel, thus giving a local
  splitting.  For dimension $3$, Bray-Brendle-Neves \cite{BrayBrenleNevesrigidity} used a different
  foliation and showed that the local splitting also holds when $M$ has nonnegative scalar
  curvature. A more systematic method in dimension $3$ is provided by Micallef-Moraru
  \cite{micallef-moraru-splitting}. Chu-Lee-Zhu \cite{chuleezhu_n_systole} applied the method of
  \cite{BrayBrenleNevesrigidity} and also proved the above lemma.

\section{Proof of the inequality}\label{section3}
    
    \begin{theorem}\label{vol inequality in Cm case}
    Assume \(3 \leq  n \leq  7\), \(1 \leq  m \leq  n-1\). Suppose $(N^{n},g)$ is a closed
      orientable $n$-dimensional Riemannian manifold with positive $m$-intermediate curvature
      $C_{m}\geq n-m>0$. Suppose there exists a smooth map
      $f:N\rightarrow E^{n-m+1}\times \mathbb{T}^{m-1}$ with nonzero degree, where $E$ is a closed
      orientable $(n-m+1)$-dimensional manifold. Then
    $$\operatorname{sys}_{n-m+1}(N,g)\leq  |\mathbb{S}^{n-m+1}|.$$ 
\end{theorem}

\begin{proof}
        The dimension $n\leq 7$ ensures the regularity of the area-minimizing hypersurfaces used
          below. If $m=1$, then $C_1=\Ric_N\geq n-1$. Since $N$ itself represents the nonzero
          fundamental class in $H_n(N;\mathbb{Z})$, we have
          $\operatorname{sys}_n(N,g)\leq \operatorname{vol}(N,g)$. The Bishop-Gromov volume
          comparison theorem gives
          $\operatorname{vol}(N,g)\leq |\mathbb{S}^n|$, and the desired estimate follows.

        We now assume $m\geq 2$. By assumption, $N$ has positive spectral
          $(0,m,n-m)$-intermediate curvature. Let $d\theta_1,\ldots,d\theta_{m-1}$ be the standard
          generators of $H^1(\mathbb{T}^{m-1};\mathbb{Z})$, pulled back to $N$ by the torus
          component of $f$. Set $\Sigma_0=N$, $u_0\equiv 1$, and $k_0=0$. We construct inductively
          a weighted slicing
          \[
              \Sigma_{m-1}\subset \Sigma_{m-2}\subset\cdots\subset \Sigma_1\subset\Sigma_0=N.
          \]
          Suppose $\Sigma_j$ has been constructed for some $0\leq j\leq m-2$, and suppose it
          admits positive spectral $(k_j,m-j,n-m)$-intermediate curvature, where
          \[
              k_j=\frac{2j}{j+1}.
          \]
          By Lemma \ref{topology inherit}, there exists a hypersurface
          $\Sigma_{j+1}\subset \Sigma_j$ minimizing 
          the weighted functional
          \[
              \mathcal{A}_{k_j}(\Sigma)=\int_{\Sigma}u_j^{k_j}\,d\mathcal{H}^{n-j-1}.
          \]
    
          Applying Theorem \ref{spectral inheritance}, and using the elementary inequality $(j>0)$
          \[
              D(n-j,m-j)\geq \frac{k_j-1}{k_j}=\frac{j-1}{2j}.
          \]
           We postpone the discussion of this inequality to the section \ref{deduction}. we obtain positive spectral
          $(k_{j+1},m-j-1,n-m)$-intermediate curvature on $\Sigma_{j+1}$, where
          \[
              k_{j+1}=\frac{4}{4-k_j}=\frac{2(j+1)}{j+2}.
          \]
          After $m-1$ steps, $\Sigma_{m-1}$ is a closed $(n-m+1)$-dimensional manifold satisfying
          positive spectral $(\frac{2m-2}{m},n-m)$-Ricci curvature.

    By construction,
    \[
        [\Sigma_{m-1}]=[N]\frown
        \big(f^*(d\theta_{1})\smile\cdots \smile f^{*}(d\theta_{m-1})\big)
        \neq 0\in H_{n-m+1}(N;\mathbb{Z}).
    \]
    Hence at least one connected component $\Gamma$ of $\Sigma_{m-1}$ represents a nonzero class in
      $H_{n-m+1}(N;\mathbb{Z})$.

    If $n-m+1\geq 3$, then it is not hard to check that 
    \begin{equation}\label{condition 1}
        \frac{2m-2}{m}\leq \frac{n-m}{n-m-1}.
    \end{equation}
    Applying Theorem \ref{volume bound under spectral Ric} to $\Gamma$ gives
    \[
        |\Gamma|\leq  |\mathbb{S}^{n-m+1}|.
    \]
    If $n-m+1=2$, then the final spectral Ricci inequality on $\Gamma$ has the form
    \[
        -\frac{2m-2}{m}\Delta_{\Gamma}u+K_{\Gamma}u\geq u.
    \]
    Dividing by $u$ and integrating gives
    \[
        |\Gamma|\leq \int_{\Gamma}K_{\Gamma}
          -\frac{2m-2}{m}\int_{\Gamma}|\nabla\log u|^2
        \leq 2\pi\chi(\Gamma)\leq 4\pi=|\mathbb{S}^2|.
    \]
    Therefore
    \[
    \operatorname{sys}_{n-m+1}(N,g)\leq |\Gamma|\leq |\mathbb{S}^{n-m+1}|.
    \]
    The proof is completed.
    \end{proof}

When $m=1$, the rigidity statement in Theorem \ref{vol inequality in Cm case rigidity} follows
  directly from the equality case in the Bishop-Gromov volume comparison theorem. Hence, hereafter to discuss the
  rigidity we only need to consider $2\leq m\leq n-1$.

\begin{proposition}\label{terminal-rigidity}
    Under the assumptions of Theorem \ref{vol inequality in Cm case}, assume $2\leq m\leq n-1$ and
      suppose the following equality holds 
    \[
        \operatorname{sys}_{n-m+1}(N,g)=|\mathbb{S}^{n-m+1}|.
    \]
    Let $\Gamma\subset N$ be a connected closed $(n-m+1)$-dimensional manifold with positive spectral
      $(\frac{2m-2}{m},n-m)$-Ricci curvature and $[\Gamma]\neq 0\in H_{n-m+1}(N)$, then the
      positive function associated with the spectral Ricci curvature condition on $\Gamma$
      is constant, and $\Gamma$ is isometric to the round sphere $\mathbb{S}^{n-m+1}$.
\end{proposition}
\begin{proof}
    Set $d=n-m+1$. 
    If $d\geq3$, then $\Gamma$ has positive spectral
      $(\frac{2m-2}{m},n-m)$-Ricci curvature, and
    \[
        \frac{2m-2}{m}\leq\frac{d-1}{d-2}.
    \]
    Theorem \ref{volume bound under spectral Ric} and $[\Gamma]\neq 0\in H_{n-m+1}(N)$ gives
    \[
        \operatorname{sys}_{d}(N,g)\leq |\Gamma|\leq |\mathbb{S}^{d}|.
    \]
    Since equality holds in the theorem, both inequalities are equalities and
      $|\Gamma|=|\mathbb{S}^{d}|$.
    The equality case in Theorem \ref{volume bound under spectral Ric} implies that the associated
      spectral function is constant and that $\Gamma$ is isometric to the round sphere of sectional
      curvature
    \[
        \frac{n-m}{d-1}=1.
    \]

    It remains to consider $d=2$, equivalently $m=n-1$. In this case the spectral Ricci
      inequality on $\Gamma$ has the form
    \[
        -\frac{2m-2}{m}\Delta_{\Gamma}u+K_{\Gamma}u\geq u.
    \]
    Dividing by $u$ and integrating gives
    \[
        |\Gamma|+\frac{2m-2}{m}\int_{\Gamma}|\nabla\log u|^2
          \leq \int_{\Gamma}K_{\Gamma}=2\pi\chi(\Gamma)\leq4\pi.
    \]
    Then $4\pi=\operatorname{sys}_{2}(N,g)\leq|\Gamma|\leq4\pi$. Hence, equality holds throughout. Hence
      $\nabla u=0$, $\chi(\Gamma)=2$, and $K_{\Gamma}=1$. Thus, $\Gamma$ is the round two-sphere.
\end{proof}

The proof of Theorem \ref{vol inequality in Cm case rigidity} in next section will proceed by upward
  induction along the weighted slicing. Proposition \ref{terminal-rigidity} above provides the initial
  step of this induction: at the terminal slice $\Sigma_{m-1}$, equality forces the final spectral
  function to be constant and the relevant component to be the round sphere.
    
    \section{Proof of the rigidity}\label{section4}
    For the rest of the paper, we focus on the proof of rigidity, i.e., Theorem \ref{vol inequality in Cm case rigidity}. 
    \subsection{Local conformal change of the metric}
    In this subsection, we will prepare the metric-deformation calculations inspired by Liu
      \cite{Liu-nonnegative-Ricci-curvature} and Chu-Lee-Zhu \cite{chuleezhu_n_systole}. We assume
      the dimension of $N$ satisfies $n\leq 7$ throughout this section. Thus the regularity of the
      minimizer is no longer an issue.

    Let $r_{\text{inj}}$ be the injectivity radius of $(N,g)$. Given any point $p\in N$ and any
      constant $r\in(0,r_{\text{inj}})$ and $t\in(0,1)$,
      we consider the following conformal metric:
    \[
    g_{p,r,t}=e^{-2t(r^2-\rho^2)^5}g,
    \]
    where
    \[
        \rho=\min\{r,\text{dist}_{g}(\cdot,p)\}.
    \]
    
    Denote by $f=t(r^2-\rho^2)^{5}$. Then we have the gradient bound
        \[
            |\nabla_{N}f|\leq 10t\rho(r^2-\rho^2)^{4},
        \]
        and the Hessian estimate
        \[
        \nabla^{2}_{N}f\geq 80t\rho^2(r^2-\rho^2)^{3}d\rho \otimes d\rho
          -10t(r^2-\rho^2)^{4}(d\rho\otimes d\rho + \rho|\nabla^{2}_{N}\rho|g).
        \]
        According to the Hessian comparison theorem, for $r$ small enough and every $0<\rho<r$, we
          have
        \begin{equation}
            \nabla^{2}_{N}f\geq 80t\rho^2(r^2-\rho^2)^{3}d\rho \otimes d\rho
              -C't(r^2-\rho^2)^4g.
        \end{equation}
        Hence, for every unit vector $e$, we obtain
        \begin{equation}\label{hess_1}
            \nabla_{N}^{2}f(e,e)\geq -C't(r^{2}-\rho^2)^4,
        \end{equation}
        and, by taking the trace of the tensor inequality above with $C''=nC'$,
        \begin{equation}\label{delta}
            \Delta_{N}f\geq 80t\rho^2(r^2-\rho^2)^3-C''t(r^2-\rho^2)^{4}.
        \end{equation}
    The following lemma shows how the intermediate curvature changes with respect to the new metric
      and gives a lower bound.

    \begin{lemma}[cf. Lemma 2.8 in \cite{chuleezhu_n_systole}]\label{C_m lower bound}
        Fix $r>0$ sufficiently small. For every $1\leq\ell\leq n-1$, there exists a constant
          $\theta\in(0,1)$
          independent of $p,r$ and $t$ such that the following holds in $B_{r}(p)\backslash
          B_{\theta r}(p)$
        \[
        e^{-2t(r^2-\rho^2)^{5}}\cdot C_{\ell}^{g_{p,r,t}}(e_{1,p,r,t},...,e_{\ell,p,r,t})\geq
          C_{\ell}(e_{1},...,e_{\ell})+75t\rho^2(r^2-\rho^2)^{3},
        \]
        where $\{e_{i}\}_{i=1}^{\ell}$ is any $g$-orthonormal $\ell$-frame in $T_xN$ and
          $\{e_{i,p,r,t}\}_{i=1}^{\ell}$ is the corresponding $g_{p,r,t}$-orthonormal
          $\ell$-frame obtained by normalizing $\{e_i\}_{i=1}^{\ell}$.
    \end{lemma}
    \begin{proof}
        Let $\{e_1,\ldots,e_n\}$ be a $g$-orthonormal frame extending
          $\{e_1,\ldots,e_\ell\}$, and let
          $e_{i,p,r,t}=e^{f}e_i$ be the corresponding $g_{p,r,t}$-unit vectors. Since
          $g_{p,r,t}=e^{-2f}g$, the conformal change formula gives, for $g$-orthonormal vectors
          $v,w$,
        \[
            \begin{aligned}
            \text{Rm}_{g_{p,r,t}}(e^fv,e^fw,e^fv,e^fw)
            &=e^{2f}\Big(\text{Rm}_{g}(v,w,v,w)+\nabla^{2}_{N}f(v,v)
              +\nabla^{2}_{N}f(w,w)\\
            &\quad +v(f)^2+w(f)^2-|\nabla_{N}f|^2\Big).
            \end{aligned}
        \]
        Summing this identity over the pairs appearing in the definition of $C_\ell(\cdots)$, we obtain
        \[
        \begin{aligned}
            &\quad e^{-2t(r^2-\rho^2)^{5}}\cdot
              C_{\ell}^{g_{p,r,t}}(e_{1,p,r,t},...,e_{\ell,p,r,t})\\
            &=C_{\ell}(e_{1},...,e_{\ell})+\ell\Delta_{N}f\\
            &\quad
              +(n-1-\ell)\sum_{i=1}^{\ell}\nabla_N^2f(e_i,e_i)\\
            &\quad
              +(n-1-\ell)\sum_{i=1}^{\ell}e_i(f)^2
              +\frac{\ell(3-2n+\ell)}{2}|\nabla_N f|^2.
        \end{aligned}
        \]
        Here the coefficient of each $\nabla_N^2f(e_i,e_i)$ with $1\leq i\leq \ell$ is
          $(n-1-\ell)+\ell=n-1$, while the coefficient of each of the remaining directions
          $e_{\ell+1},\ldots,e_n$ is $\ell$; this is why the Hessian terms can be written as above.

        Since $n-1-\ell\geq0$, the terms
          $(n-1-\ell)\sum_{i=1}^{\ell}e_i(f)^2$ are nonnegative. Moreover, by \eqref{hess_1} and
          the gradient bound $|\nabla_N f|\leq 10t\rho(r^2-\rho^2)^4$, after decreasing $r$ if
          necessary, there is a constant $C_1$ depending only on $n$ such that
        \[
            (n-1-\ell)\sum_{i=1}^{\ell}\nabla_N^2f(e_i,e_i)
            +\frac{\ell(3-2n+\ell)}{2}|\nabla_N f|^2
            \geq -C_1t(r^2-\rho^2)^4.
        \]
        Therefore
        \[
        \begin{aligned}
            &\quad e^{-2t(r^2-\rho^2)^{5}}\cdot
              C_{\ell}^{g_{p,r,t}}(e_{1,p,r,t},...,e_{\ell,p,r,t})\\
            &\geq C_{\ell}(e_{1},...,e_{\ell})+\ell\Delta_N f-C_{1}t(r^2-\rho^2)^{4}\\
            &\geq C_{\ell}(e_{1},...,e_{\ell})+80t\rho^2(r^2-\rho^2)^{3}
              -C_{2}t(r^2-\rho^2)^{4},
        \end{aligned}
        \]
        where in the last line we used \eqref{delta} and $\ell\geq1$.

        It remains to absorb the last error term in the annulus. If
          $x\in B_r(p)\backslash B_{\theta r}(p)$, then $\rho(x)\geq\theta r$ and
          $r^2-\rho(x)^2\leq (1-\theta^2)r^2$. Choosing $\theta$ sufficiently close to $1$, depending
          only on $C_2$, gives
        \[
            C_2(r^2-\rho^2)\leq 5\rho^2
        \]
        in this annulus. Hence
        \[
        \begin{aligned}
            &\quad e^{-2t(r^2-\rho^2)^{5}}\cdot
              C_{\ell}^{g_{p,r,t}}(e_{1,p,r,t},...,e_{\ell,p,r,t})\\
            &\geq C_{\ell}(e_{1},...,e_{\ell})+75t\rho^2(r^2-\rho^2)^{3}.
        \end{aligned}
        \]
    \end{proof}
    The following lemma gives an upper bound for the intermediate curvature of a hypersurface under
      the new metric. 
    \begin{lemma}[cf. Lemma 2.9 in \cite{chuleezhu_n_systole}]\label{C_m-1_upper}
        Given $2\leq \ell\leq n-1$, a positive constant $\Lambda$ and a sufficiently small $r$, there
          exist $\iota>0$ and $\theta\in(0,1)$ depending on $\Lambda$ such that if $\Sigma$ is a
          smooth embedded hypersurface satisfying $|A_\Sigma|_g\leq \Lambda$ and
          $|\nabla_{g}^{\Sigma}\rho|\leq \iota$ in $B_{r}(p)$, then the following holds in
          $(B_{r}(p)\backslash B_{\theta r}(p))\cap \Sigma$
        \[
            e^{-2t(r^{2}-\rho^2)^5}\cdot
              C_{\ell-1}^{g_{p,r,t}}(e_{1,p,r,t},...,e_{\ell-1,p,r,t})\leq
              C_{\ell-1}^{g}(e_{1},...,e_{\ell-1})+50 t \rho^2(r^2-\rho^2)^{3},
        \]
        where $\{e_{i}\}_{i=1}^{\ell-1}$ is any $g$-orthonormal $(\ell-1)$-frame in
          $T_x\Sigma$ and $\{e_{i,p,r,t}\}_{i=1}^{\ell-1}$ is the corresponding
          $g_{p,r,t}$-orthonormal $(\ell-1)$-frame obtained by normalizing
          $\{e_i\}_{i=1}^{\ell-1}$. 
    \end{lemma}

    Before the proof, we clarify that $C_{\ell-1}^{g}(e_{1},...,e_{\ell-1})$ is intrinsic $(\ell-1)$-intermediate curvature of $\Sigma$ with the induced metric from $(N,g)$. 
    \begin{proof}
        Let $h=g|_{\Sigma}$ and $\tilde h=g_{p,r,t}|_{\Sigma}=e^{-2f}h$, where
          $f=t(r^2-\rho^2)^5$. We work at a point $x\in
          (B_r(p)\backslash B_{\theta r}(p))\cap\Sigma$ and extend
          $\{e_1,\ldots,e_{\ell-1}\}$ to an $h$-orthonormal frame
          $\{e_1,\ldots,e_{n-1}\}$ of $T_x\Sigma$. Then
          $e_{i,p,r,t}=e^f e_i$.

        First we record the estimates for the restriction of $f$ to $\Sigma$. Since
          $|\nabla^\Sigma\rho|\leq\iota$,
        \[
            |\nabla^\Sigma f|\leq 10t\iota\rho(r^2-\rho^2)^4.
        \]
        Moreover,
        \[
            \nabla_\Sigma^2f
            =(80t\rho^2(r^2-\rho^2)^3-10t(r^2-\rho^2)^4)
              d^\Sigma\rho\otimes d^\Sigma\rho
              -10t\rho(r^2-\rho^2)^4\nabla_\Sigma^2\rho .
        \]
        The Hessian comparison theorem in $N$, together with $|A_\Sigma|_g\leq\Lambda$, gives
          $\rho|\nabla_\Sigma^2\rho|\leq C$ for $r$ sufficiently small, where $C$ depends only on
          $N$ and $\Lambda$. Hence
        \begin{equation}\label{sigma-hessian-upper}
            |\nabla_\Sigma^2f|+|\Delta_\Sigma f|
            \leq C t\iota^2\rho^2(r^2-\rho^2)^3+C t(r^2-\rho^2)^4.
        \end{equation}

        For $h$-orthonormal tangent vectors $v,w\in T_x\Sigma$, the conformal change formula for
          $\tilde h=e^{-2f}h$ gives
        \[
        \begin{aligned}
            \operatorname{Rm}_{\tilde h}(e^fv,e^fw,e^fv,e^fw)
            &=e^{2f}\Big(\operatorname{Rm}_{h}(v,w,v,w)+\nabla_\Sigma^2f(v,v)
              +\nabla_\Sigma^2f(w,w)\\
            &\quad +v(f)^2+w(f)^2-|\nabla^\Sigma f|^2\Big).
        \end{aligned}
        \]
        Summing this identity we obtain
        \[
        \begin{aligned}
            &\quad e^{-2t(r^2-\rho^2)^5}\cdot
              C_{\ell-1}^{g_{p,r,t}}(e_{1,p,r,t},...,e_{\ell-1,p,r,t})\\
            &=C_{\ell-1}^{g}(e_{1},...,e_{\ell-1})+(\ell-1)\Delta_\Sigma f\\
            &\quad +(n-\ell-1)\sum_{i=1}^{\ell-1}\nabla_\Sigma^2f(e_i,e_i)
              +(n-\ell-1)\sum_{i=1}^{\ell-1}e_i(f)^2\\
            &\quad +\frac{(\ell-1)(\ell+4-2n)}{2}|\nabla^\Sigma f|^2 .
        \end{aligned}
        \]
        The last coefficient is nonpositive for $2\leq\ell\leq n-1$, and the remaining gradient
          square terms are controlled by the preceding gradient estimate. Combining this with
          \eqref{sigma-hessian-upper}, after decreasing $r$ if necessary, yields
        \[
        \begin{aligned}
            &\quad e^{-2t(r^2-\rho^2)^5}\cdot
              C_{\ell-1}^{g_{p,r,t}}(e_{1,p,r,t},...,e_{\ell-1,p,r,t})\\
            &\leq C_{\ell-1}^{g}(e_{1},...,e_{\ell-1})
              +C_1t\iota^2\rho^2(r^2-\rho^2)^3+C_2t(r^2-\rho^2)^4 .
        \end{aligned}
        \]
        We now choose $\iota>0$ so small that $C_1\iota^2\leq25$. After that, choose
          $\theta\in(0,1)$ sufficiently close to $1$ so that
          $C_2(r^2-\rho^2)\leq25\rho^2$ on $B_r(p)\backslash B_{\theta r}(p)$. The desired estimate
          follows.
    \end{proof}

    \subsection{Splitting of $\Sigma_{k}$}\label{splitting-of-Sigma1}\label{deduction}
    In this subsection, we use metric-deformation to prove Theorem \ref{vol inequality in Cm case rigidity}.
    We first record the numerical inequality needed below. Assume either
      $3\leq n\leq 6$, $1\leq m\leq n-1$, or $n=7$ and $m\in\{1,2,5,6\}$. Then, for every
      $1\leq k\leq m-2$, the pair $(n-k,m-k)$ satisfies the hypotheses in Chen's algebraic
      inequality and
    \[
        D(n-k,m-k)>\frac{k-1}{2k}.
    \]
    Indeed, the first term in the definition of $D(n-k,m-k)$ is greater than $\frac12$, and the
      second term is greater than $\frac{k-1}{2k}$ in the above range. For the third term, writing
      $N_k=n-k$ and $M_k=m-k$, we have
    \[
    \frac{M_k^2-M_kN_k+M_k+N_k}{2(M_k^2-M_kN_k+2N_k-2)}
      -\frac{k-1}{2k}
      =
      \frac{m^2-mn+2n-2}{2k(M_k^2-M_kN_k+2N_k-2)}>0.
    \]
    The excluded cases $n=7$, $m=3,4$ are precisely the cases where this strict inequality fails:
    \[
    \begin{array}{c|c|c|c}
        (n,m) & k & D(n-k,m-k) & \dfrac{k-1}{2k} \\
        \hline
        (7,3) & 1 & D(6,2)=0 & 0 \\
        (7,4) & 1 & D(6,3)=0 & 0 \\
        (7,4) & 2 & D(5,2)=\dfrac14 & \dfrac14 .
    \end{array}
    \]
   But nonstrict inequality always holds in these ranges.
\begin{proof}[Proof of Theorem \ref{vol inequality in Cm case rigidity}]
If $m=1$, then $C_1=\Ric_N\geq n-1$ and equality in \eqref{maininequality} is equality in the
  Bishop-Gromov volume comparison. Hence $N$ is isometric to the round sphere $\mathbb{S}^n$, and
  the conclusion follows. We therefore assume $2\leq m\leq n-1$ below.
  
Since $N$ admits $f:N\rightarrow E\times \mathbb{T}^{m-1}$ with nonzero degree, then $[N]\frown (f^{*}(d\theta_{1})\smile f^{*}(d\theta_{2}) \cdots \smile f^{*}(d\theta_{m-1}))\neq 0\in H_{n-m+1}(N)$.
We use induction to prove that for any $0 \leq k\leq m-1$ and
    every $(n-k)$-dimensional closed submanifold $W^{n-k}$ in $N$ that satisfies 
    spectral $(\frac{2k}{k+1},m-k,n-m)$-intermediate curvature and $W^{n-k}$ in the homology class $[N]\frown (f^{*}(d\theta_{1})\smile f^{*}(d\theta_{2}) \cdots \smile f^{*}(d\theta_{k}))\neq 0\in H_{n-k}(N)$ (if $k=0$, take homology class $[N]$), the following
    hold:
    \begin{enumerate}[label=(\roman*)]
        \item $W^{n-k}$ is isometrically covered by
          $\mathbb{S}^{n-m+1}\times\mathbb{R}^{m-k-1}$.
        \item The positive function $u_W$ associated with the spectral curvature condition on $W^{n-k}$ is constant.
    \end{enumerate}

 If this is proved, the case $k=0$ gives the desired conclusion for $N$ and completes the proof
   of Theorem \ref{vol inequality in Cm case rigidity}.

Proposition \ref{terminal-rigidity} gives the initial case $k=m-1$.

For the induction step, fix $1\leq k\leq m-2$. Assume the conclusion is known for $k+1$; we prove
  it for $k$. Since $W^{n-k}$ satisfies spectral
  $(\frac{2k}{k+1},m-k,n-m)$-intermediate curvature, there exists
  $u_W\in C^{\infty}(W^{n-k})$ such that
  \[
      -\frac{2k}{k+1}\Delta_{W}u_W+C_{m-k}u_{W}\geq (n-m)u_{W}.
  \]
  By Lemma \ref{topology inherit}, choose a closed
  hypersurface $\Gamma\subset W$ minimizing
  $\int u_{W}^{\frac{2k}{k+1}}dA$ and $\Gamma$ in the homology class
  $[N]\frown (f^{*}(d\theta_{1})\smile f^{*}(d\theta_{2}) \cdots \smile f^{*}(d\theta_{k+1}))\neq 0\in H_{n-k-1}(N)$. Theorem \ref{spectral inheritance} implies that $\Gamma$ satisfies
positive spectral
  $(\frac{2k+2}{k+2},m-k-1,n-m)$-intermediate curvature. Hence the induction hypothesis gives the
  covering conclusion and the constancy of the corresponding spectral function on
  $\Gamma$. By the numerical observation at the beginning of Subsection
  \ref{splitting-of-Sigma1}, we have
  $D(n-k,m-k)>\frac{k-1}{2k}$.

  \vskip.2cm
  \textbf{Step 1: there is a sequence of totally geodesic hypersurfaces $\Gamma^{(i)}\subset W$ smoothly converging to $\Gamma$ from one side.}
\vskip.1cm
In order to guarantee that $\Gamma^{(i)}$ converges from one side of $\Gamma$,
          we cut $W$ along the hypersurface $\Gamma$ and then consider the metric
          completion $(\hat{W},\hat{g})$ of the complement $(W\backslash
          \Gamma,g)$.  Note
          \[
              [\Gamma]=[W]\frown f^*(d\theta_{k+1})\neq 0
              \in H_{n-k-1}(W).
          \]
          Hence,
          $\Gamma$ is non-separating in
          $W$ and $\partial \hat{W}$ consists of two copies of $\Gamma$. We
          fix one component of $\partial \hat{W}$ and denote it by $\hat{\Gamma}$.
          For $\epsilon$ small, we take a minimizing arclength parametrized geodesic
          $\hat{\gamma}:[0,\epsilon)\rightarrow (\hat{W},\hat{g})$ with
          $\hat{\gamma}(0)\in\hat{\Gamma}$ and $\hat{\gamma}\perp \hat{\Gamma}$. Note
          that we can take $\epsilon$ sufficiently small such that
          $\operatorname{dist}(\hat{\gamma}(t),\hat{\Gamma})=t$ for $t<\epsilon$. Denote
          $\hat{p}=\hat{\gamma}(s)$ for some fixed $s\in(0,\epsilon)$ and
          $\hat{d}=\operatorname{dist}_{\hat{g}}(\cdot,\hat{\Gamma})$. For $a>0$, set
          \[
              \hat B_a(\hat p)=\{x\in\hat{W}:\operatorname{dist}_{\hat g}(x,\hat p)<a\}.
          \]
          Let
          $\hat{\nu}_{\Gamma}$ be the unit outer normal vector of $\hat{\Gamma}$ in
          $\hat{W}$.
        
        Fix $\theta\in(0,1)$ and $\iota\in(0,1)$ such that Lemma \ref{C_m lower bound} and Lemma
          \ref{C_m-1_upper} hold with the choice $\Lambda = \max |A^{\hat{\Gamma}}|+1$. Since
          $\langle\nabla_{\hat{W}}\hat{d},\hat{\nu}_{\Gamma}\rangle|_{\hat{\gamma}(0)} =1$,
          continuity guarantees that there is an open neighborhood $U$ of $\hat{\gamma}(0)$ in the
          boundary component $\hat{\Gamma}$ such that
          $\langle\nabla_{\hat{W}}\hat{d},\hat{\nu}_{\Gamma}\rangle>\sqrt{1-\iota^2}$ on
          $U$. Then there exists $\delta\in(0,s)$ small enough such that
        \[
        \langle\nabla_{\hat{W}}\hat{d},\hat{\nu}_{\Gamma}\rangle>\sqrt{1-\iota^2}
          \quad\text{in}\  \hat{B}_{s+\delta}(\hat{p})\cap \hat{\Gamma}.
        \]
        
        Let $r=s+\tau$ with $0<\tau<\delta$ to be determined later. We consider the conformal metric
        \[
            \hat{g}_{r,t}=e^{-2t(r^2-\hat{\rho}^2)^5}\hat{g}
        \]
        where 
        \[
            \hat{\rho}=\min\{r,\text{dist}_{\hat{g}}(\cdot,\hat{p})\}.
        \]
        Set $\phi=t(r^2-\hat{\rho}^2)^5$. Since $\hat{\Gamma}$ is totally geodesic in
          $(\hat{W},\hat g)$, its $\hat g$-mean curvature is zero. Under the conformal change
          $\hat g_{r,t}=e^{-2\phi}\hat g$, the mean curvature of $\hat{\Gamma}$ with respect
          to the $\hat g_{r,t}$-unit outer normal $e^\phi\hat\nu_{\Gamma}$ is
        \[
            \hat H_{r,t}=e^\phi\big(\hat H-(n-k-1)\hat\nu_{\Gamma}(\phi)\big)
              =-e^\phi(n-k-1)\hat\nu_{\Gamma}(\phi).
        \]
        Moreover,
        \[
            \hat\nu_{\Gamma}(\phi)
            =-10t(r^2-\hat\rho^2)^4\hat\rho
              \langle\nabla_{\hat g}\hat\rho,\hat\nu_{\Gamma}\rangle.
        \]
        Therefore
        \[
            \hat{H}_{r,t}
            =10(n-k-1)e^{t(r^2-\hat{\rho}^2)^5}t(r^2-\hat{\rho}^2)^{4}\hat{\rho}
              \langle\nabla_{\hat{g}}\hat{\rho},\hat{\nu}_{\Gamma}\rangle \geq0.
        \]
        If we further consider the conformal metric
          $\tilde{g}_{r,t}=u_{W}^{\frac{2k}{(n-k-1)(k+1)}}\hat{g}_{r,t}$, then the mean curvature
          of $\hat{\Gamma}$ in $(\hat{W},\tilde{g}_{r,t})$ with respect to the unit
          outer normal is
        \[
            \tilde{H}_{r,t}
            =u_{W}^{-\frac{k}{(n-k-1)(k+1)}}
              \left(\hat{H}_{r,t}+\frac{k}{k+1}u_W^{-1}\nu_{r,t}(u_W)\right),
        \]
        where $\nu_{r,t}=e^{t(r^2-\hat\rho^2)^5}\hat\nu_{\Gamma}$ is the $\hat g_{r,t}$-unit outer
          normal. Since $\hat{\Gamma}$ satisfies induction assumptions, the rigidity part of Theorem \ref{spectral inheritance} implies
          $\hat{\nu}_{\Gamma}(u_{W})=0$ on $\hat{\Gamma}$. Hence
        \[
            \tilde{H}_{r,t}
            =u_{W}^{-\frac{k}{(n-k-1)(k+1)}}\hat{H}_{r,t}\geq 0.
        \]
        Thus,
        it follows from geometric measure theory that there exists a smooth embedded hypersurface
          $\hat{\Gamma}_{t,r}\subset\hat{W}$ such that
        \[
            \int_{\hat{\Gamma}_{t,r}}u_{W}^{\frac{2k}{k+1}} dA_{\hat{g}_{r,t}} =
              \min_{\substack{ \Sigma\subset\hat{W}\\ \Sigma \ \text{homologous to}\
              \hat{\Gamma}
              }}\int_{\Sigma}u_{W}^{\frac{2k}{k+1}}dA_{\hat{g}_{r,t}}.
        \]
        
        We claim that $\hat{\Gamma}_{t,r}\cap \hat{B}_{r}(\hat{p})\neq \emptyset$. Otherwise, because
          $\hat{g}_{r,t}=\hat{g}$ outside $\hat{B}_{r}(\hat{p})$ and $\hat{g}_{r,t}<\hat{g}$ inside
          $\hat{B}_{r}(\hat{p})$, we have
        \begin{align*}
            \int_{\hat{\Gamma}_{t,r}}u_{W}^{\frac{2k}{k+1}} dA_{\hat{g}_{r,t}}&\leq
              \int_{\hat{\Gamma}}u_{W}^{\frac{2k}{k+1}}
              dA_{\hat{g}_{r,t}}\\&<\int_{\hat{\Gamma}}u_{W}^{\frac{2k}{k+1}} dA_{\hat{g}}\\
            &\leq\int_{\hat{\Gamma}_{t,r}}u_{W}^{\frac{2k}{k+1}} dA_{\hat{g}}\\
    &=\int_{\hat{\Gamma}_{t,r}}u_{W}^{\frac{2k}{k+1}}dA_{\hat{g}_{r,t}}
        \end{align*}
        which is a contradiction.
          
        Notice as $t\rightarrow 0^+$,  $\hat{g}_{r,t}\rightarrow \hat{g}$ in $C^{4,\alpha}$.
          Moreover, for any small $t$,
        \[\int_{\hat{\Gamma}_{t,r}}u_{W}^{\frac{2k}{k+1}} dA_{\hat{g}_{r,t}}\leq
          \int_{\hat{\Gamma}}u_{W}^{\frac{2k}{k+1}} dA_{\hat{g}}. \]
        Since $\hat{\Gamma}_{t,r}$ is a stable minimal hypersurface with respect to the new
          metric $\tilde g_{r,t}$, the curvature estimate for stable minimal hypersurfaces applies.
          The metrics $\tilde g_{r,t}$ converge smoothly as $t\rightarrow0^+$, so the constants in
          the estimate are uniform for $t$ sufficiently small. Hence, after passing to a subsequence,
          $\hat{\Gamma}_{t,r}$ converges to a smooth weighted minimal hypersurface
          $\hat{\Gamma}_{r}$.
        Thus
        \[\int_{\hat{\Gamma}_{r}}u_{W}^{\frac{2k}{k+1}} dA_{\hat{g}}\leq
          \int_{\hat{\Gamma}}u_{W}^{\frac{2k}{k+1}} dA_{\hat{g}}.\]
        On the other hand, by the construction, we have
          $[\hat{\Gamma}_{r}]=[\hat{\Gamma}]\in H_{n-k-1}(\hat{W})$. Since $\hat{\Gamma}$ minimizes
          $\int u_W^{\frac{2k}{k+1}}dA_{\hat g}$ in this homology class, the reverse inequality also
          holds. Hence equality in above inequality holds, and $\hat{\Gamma}_r$ is again a weighted area-minimizer in
          the same homology class. In particular,
          the rigidity part of Theorem \ref{spectral inheritance} implies that
          $\hat{\Gamma}_{r}$ is totally geodesic. Since
          $\hat{\Gamma}_{t,r}\cap \hat{B}_{r}(\hat{p})\neq \emptyset$, then
          $\hat{\Gamma}_{r}$ must intersect $\overline{\hat{B}_{r}(\hat{p})}$. 

        Next, we need to show $\hat{\Gamma}_{r}\neq \hat{\Gamma}$. If not, we have
          $\hat{\Gamma}_{t,r}\rightarrow\hat{\Gamma}$ as $t\rightarrow 0^+$. 
        Making use of the fact that $\hat{\Gamma}_{t,r}$ is a stable weighted minimal
          hypersurface in $\hat{W}$ with respect to the functional
          $\int_{\Sigma}u_{W}^{\frac{2k}{k+1}}d A_{\hat{g}_{r,t}}$, using a computation similar
          to that in Theorem \ref{spectral inheritance}, we have
    \begin{equation}
    \begin{split}
    0&\leq  \int_{\hat{\Gamma}_{t,r}}\bigg(
      \frac{2k+2}{k+2}|\hat\nabla^{\hat{\Gamma}_{t,r}}_{\hat{g}_{r,t}}\phi|^2+\big(\frac{2k}{k+1}u_{W}^{-1}\hat{\Delta}^{\hat{W}}_{\hat{g}_{r,t}} u_{W}-|\hat{A}^{\hat{\Gamma}_{t,r}}_{\hat{g}_{r,t}}|^2\\
    & -\Ric^{\hat{W}}_{\hat{g}_{r,t}}(\hat{\nu}_{r,t},\hat{\nu}_{r,t})+\frac{k-1}{2k}|H^{\hat{\Gamma}_{t,r}}_{\hat{g}_{r,t}}|^2\big)\phi^2 \bigg) dA_{\hat{g}_{r,t}}.
    \end{split}
    \end{equation}
    Moreover, the Gauss equation yields
\begin{align*}
    C_{m-k}^{\hat{W},\hat{g}_{r,t}}\leq
      C_{m-k-1}^{\hat{\Gamma}_{t,r},\hat{g}_{r,t}}-\sum_{i=2}^{m-k}\sum_{j=i+1}^{n-k}
      (\hat{h}_{ii}\hat{h}_{jj}-\hat{h}_{ij}^2)+\Ric^{\hat{W}}_{\hat{g}_{r,t}}(\hat{\nu}_{r,t},\hat{\nu}_{r,t})
\end{align*}    
where $\hat h(\cdot,\cdot)$ is the second fundamental form of
  $\hat{\Gamma}_{t,r}\subset(\hat{W},\hat g_{r,t})$ with respect to the unit normal
  $\hat\nu_{r,t}$. Now we have
    \begin{equation}
    \begin{split}
    0&\leq  \int_{\hat{\Gamma}_{t,r}}\bigg(
      \frac{2k+2}{k+2}|\hat\nabla^{\hat{\Gamma}_{t,r}}_{\hat{g}_{r,t}}\phi|^2+\big(\frac{2k}{k+1}u_{W}^{-1}\hat{\Delta}^{\hat{W}}_{\hat{g}_{r,t}} u_{W}-|\hat{A}^{\hat{\Gamma}_{t,r}}_{\hat{g}_{r,t}}|^2\\
    &-C_{m-k}^{\hat{W},\hat{g}_{r,t}}+C_{m-k-1}^{\hat{\Gamma}_{t,r},\hat{g}_{r,t}}-\sum_{i=2}^{m-k}\sum_{j=i+1}^{n-k} (\hat{h}_{ii}\hat{h}_{jj}-\hat{h}_{ij}^2)+\frac{k-1}{2k}|H^{\hat{\Gamma}_{t,r}}_{\hat{g}_{r,t}}|^2\big)\phi^2 \bigg) dA_{\hat{g}_{r,t}}.
    \end{split}
    \end{equation}

As in the proof of Theorem \ref{spectral inheritance}, we use the sharp algebraic inequality
  proved by Chen \cite[Lemma 5.8]{chenshuli_end}, now applied with the parameters $n-k$ and
  $m-k$:
\[|\hat{A}^{\hat{\Gamma}_{t,r}}_{\hat{g}_{r,t}}|^2+\sum_{i=2}^{m-k}\sum_{j=i+1}^{n-k}
  (\hat{h}_{ii}\hat{h}_{jj}-\hat{h}_{ij}^2)\geq
  D(n-k,m-k)|H^{\hat{\Gamma}_{t,r}}_{\hat{g}_{r,t}}|^{2}\]
provided
\[
    (m-k)^2-(m-k)(n-k)+2(n-k)-2>0
\]
and
\[
    (m-k)^2-(m-k)(n-k)+(m-k)+(n-k)\geq 0.
\]

Since $D(n-k,m-k)>\frac{k-1}{2k}$ by the numerical observation before the proposition, the mean
  curvature term is nonpositive. Using the above inequality, we have
\begin{equation}
    \begin{split}
    0&\leq  \int_{\hat{\Gamma}_{t,r}}\bigg(
      \frac{2k+2}{k+2}|\hat\nabla^{\hat{\Gamma}_{t,r}}_{\hat{g}_{r,t}}\phi|^2+\big(\frac{2k}{k+1}u_{W}^{-1}\hat{\Delta}^{\hat{W}}_{\hat{g}_{r,t}} u_{W}-C_{m-k}^{\hat{W},\hat{g}_{r,t}}+C_{m-k-1}^{\hat{\Gamma}_{t,r},\hat{g}_{r,t}}\big)\phi^2 \bigg) dA_{\hat{g}_{r,t}}.
    \end{split}
    \end{equation}
    Our goal is to write the above inequality back in the original metric.
    To do so, we first prepare the following estimates. In
      $B_{r}(\hat{p})\backslash B_{\theta r}(\hat{p})$ when $\theta$ is sufficiently close to $1$,
    \[
    \begin{aligned}
    &\quad\hat{\Delta}^{\hat{W}}_{\hat{g}_{r,t}} u_{W}\\
    &=
      e^{2t(r^2-\hat{\rho}^2)^5}\big(\Delta_{\hat{g}}^{\hat{W}}u_{W}-(n-k-2)\langle\nabla^{\hat{W}}_{\hat{g}}(t(r^2-\hat{\rho}^2)^5),\nabla_{\hat{g}}^{\hat{W}}u_{W}\rangle\big)\\
    &\leq
      e^{2t(r^2-\hat{\rho}^2)^5}\big(\frac{k+1}{2k}(C^{\hat{W},\hat{g}}_{m-k}-(n-m))u_{W}-(n-k-2)\langle\nabla^{\hat{W}}_{\hat{g}}(t(r^2-\hat{\rho}^2)^5),\nabla_{\hat{g}}^{\hat{W}}u_{W}\rangle\big)\\
    &\leq
      e^{2t(r^2-\hat{\rho}^2)^5}\big(\frac{k+1}{2k}(C^{\hat{W},\hat{g}}_{m-k}-(n-m))u_{W}
        +Ct\hat\rho(r^2-\hat\rho^2)^{4}\big),
    \end{aligned}
    \]
    where the last line uses the uniform boundedness of $|\hat\nabla u_W|$ and does not require a
      sign for the gradient term. Moreover, on $\hat{\Gamma}_{t,r}$,
    \[
        dA_{\hat g_{r,t}}
        =e^{-(n-k-1)t(r^2-\hat\rho^2)^5}dA_{\hat g},
    \]
    and
    \[
    |\hat\nabla^{\hat{\Gamma}_{t,r}}_{\hat{g}_{r,t}}\phi|^2=
      e^{2t(r^2-\hat{\rho}^2)^5}|\hat\nabla^{\hat{\Gamma}_{t,r}}_{\hat{g}}\phi|^2.
    \]
    Hence the gradient term together with the area element carries the factor
    \[
        e^{(2-(n-k-1))t(r^2-\hat\rho^2)^5}.
    \]
    We now combine these estimates with Lemma \ref{C_m lower bound} and Lemma
      \ref{C_m-1_upper}. Since
    \[
        e^{-2t(r^2-\hat{\rho}^2)^5}C_{m-k}^{\hat{W},\hat g_{r,t}}
        \geq C_{m-k}^{\hat{W},\hat g}+75t\hat\rho^2(r^2-\hat\rho^2)^3
    \]
    and
    \[
        e^{-2t(r^2-\hat{\rho}^2)^5}C_{m-k-1}^{\hat{\Gamma}_{t,r},\hat g_{r,t}}
        \leq C_{m-k-1}^{\hat{\Gamma}_{t,r},\hat g}
          +50t\hat\rho^2(r^2-\hat\rho^2)^3,
    \]
    the curvature terms in the preceding inequality satisfy
    \[
    \begin{aligned}
    &\quad \frac{2k}{k+1}u_W^{-1}\hat\Delta^{\hat{W}}_{\hat g_{r,t}}u_W
      -C_{m-k}^{\hat{W},\hat g_{r,t}}
      +C_{m-k-1}^{\hat{\Gamma}_{t,r},\hat g_{r,t}}\\
    &\leq e^{2t(r^2-\hat{\rho}^2)^5}\big(C_{m-k-1}^{\hat{\Gamma}_{t,r},\hat g}
      -(n-m)-25t\hat\rho^2(r^2-\hat\rho^2)^3
      +Ct\hat\rho(r^2-\hat\rho^2)^4\big).
    \end{aligned}
    \]
    On the annulus, by taking $\theta$ closer to $1$ if necessary, the error term
      $Ct\hat\rho(r^2-\hat\rho^2)^4$ is bounded by a small multiple of
      $t\hat\rho^2(r^2-\hat\rho^2)^3$. Also, since $n-k-1\geq2$, the factor
      $e^{(2-(n-k-1))t(r^2-\hat\rho^2)^5}$ in the gradient term is at most $1$. The same factor
      appears in the zeroth-order terms after the curvature estimates above; its deviation from
      $1$ is bounded by $Ct(r^2-\hat\rho^2)^5$, which is also absorbed into
      $t\hat\rho^2(r^2-\hat\rho^2)^3$ on the annulus. Thus, for all sufficiently small positive
      $t$, the preceding inequality implies
    \begin{equation}\label{contradiction-inequality}  \begin{split}
    0&\leq  \int_{\hat{\Gamma}_{t,r}}\bigg(
      \frac{2k+2}{k+2}|\hat\nabla^{\hat{\Gamma}_{t,r}}_{\hat{g}}\phi|^2+\big(C_{m-k-1}^{\hat{\Gamma}_{t,r},\hat{g}}-(n-m)-24t\hat{\rho}^2(r^2-\hat{\rho}^2)^{3}\big)\phi^2 \bigg) dA_{\hat{g}}.
    \end{split}
    \end{equation}
       Thus $\hat{\Gamma}_{t,r}$ satisfies nonnegative spectral
         $(\frac{2k+2}{k+2},m-k-1,n-m)$-intermediate curvature in the original metric $\hat g$.
         Moreover, note that $\hat{\Gamma}_{t,r}$ lies in the same nonzero capped homology class as
         $\hat{\Gamma}$. Hence, the induction hypothesis applies to
         $S=\hat{\Gamma}_{t,r}$. In particular, $S$ is isometrically covered by
         $\mathbb{S}^{n-m+1}\times\mathbb{R}^{m-k-2}$, and hence
         $C_{m-k-1}^{\hat{\Gamma}_{t,r},\hat{g}}=n-m$.
        Then \eqref{contradiction-inequality} becomes
        \begin{equation*}
        0\leq  \int_{\hat{\Gamma}_{t,r}}\bigg(
          \frac{2k+2}{k+2}|\hat\nabla^{\hat{\Gamma}_{t,r}}_{\hat{g}}\phi|^2-24t\hat{\rho}^2(r^2-\hat{\rho}^2)^{3}\phi^2 \bigg) dA_{\hat{g}}.
        \end{equation*}
        Under the contradiction assumption, $\hat{\Gamma}_{t,r}$ converges smoothly to
          $\hat{\Gamma}$ as $t\to0^+$. Since $\hat{\Gamma}$ passes through the annulus
          $B_r(\hat p)\backslash B_{\theta r}(\hat p)$, for $t$ sufficiently small the hypersurface
          $\hat{\Gamma}_{t,r}$ intersects this annulus in a set of positive measure. Taking
          $\phi=1$, the right-hand side is therefore strictly negative, a contradiction. Thus we
          have proved that $\hat{\Gamma}_{r}\neq \hat{\Gamma}$. 

            Since the metric $\hat{g}_{r,t}$ converges to $\hat{g}$ in $C^{4,\alpha}$ as
              $t\to0^+$, the limiting hypersurface $\hat{\Gamma}_{r}$ converges smoothly to
              $\hat{\Gamma}$ as $r\to0^+$. To see that $r$ can be chosen to tend to zero, take
              $s_i\to0^+$ and choose $\tau_i\in(0,\delta_i)$; then $r_i=s_i+\tau_i\to0^+$. Hence,
              we obtain a sequence of totally geodesic hypersurfaces
              $\Gamma^{(i)}=\hat{\Gamma}_{r_{i}}\subset W$ smoothly converging to
              $\Gamma$ from one side as $i\rightarrow \infty.$
            
\vskip.2cm
  \textbf{Step 2: for any point $p\in W$ we can find a
      closed hypersurface $\Gamma(p)\subset W$ passing through $p$ which minimizes
      $\int u_{W}^{\frac{2k}{k+1}} dA$ in the class of hypersurfaces homologous to
      $\Gamma$.}
\vskip.1cm
Let $\mathcal{P}$ be the set of points $p\in W$ for which there exists a closed
      hypersurface $\Gamma(p)\subset W$ passing through $p$ and minimizing
      $\int u_W^{\frac{2k}{k+1}}dA$ in the class of hypersurfaces homologous to
      $\Gamma$. This set is nonempty, since it contains $\Gamma$.

    We first note that $\mathcal{P}$ is closed. Indeed, if $p_i\in\mathcal{P}$ and
      $p_i\to p$, choose corresponding minimizing hypersurfaces
      $\Gamma(p_i)$. These hypersurfaces have uniformly bounded weighted area and are stable
      weighted minimizers. By the compactness theorem for stable minimizing hypersurfaces, after
      passing to a subsequence they converge smoothly to a minimizing hypersurface
      $\Gamma(p)$ in the same homology class. Since $p_i\in\Gamma(p_i)$ and
      $p_i\to p$, the limit hypersurface passes through $p$. Hence $p\in\mathcal{P}$.

    Suppose, toward a contradiction, that $\mathcal{P}\neq\ W$. Then
      $W \setminus\mathcal{P}$ is a nonempty open set. Choose a boundary point
      $p\in\partial( W\setminus\mathcal{P})\subset\mathcal{P}$ and let
      $\Gamma(p)$ be a minimizing hypersurface passing through $p$. Applying the argument of
      \textbf{Step 1} with $\Gamma(p)$ in place of $\Gamma$, we obtain a sequence of
      minimizing hypersurfaces homologous to $\Gamma(p)$, and hence to $\Gamma$,
      converging smoothly to $\Gamma(p)$ from the side facing
      $W\setminus\mathcal{P}$. For $i$ sufficiently large, one of these hypersurfaces meets
      $W\setminus\mathcal{P}$. But every point on such a hypersurface belongs to
      $\mathcal{P}$ by definition, a contradiction. Therefore $\mathcal{P}=W$.

      \vskip.2cm
  \textbf{Step 3: the function $u_W$ is constant on $W$.
    Moreover, $W$ is isometrically covered by $\mathbb{S}^{n-m+1}\times\mathbb{R}^{m-k-1}$.}
\vskip.1cm
By \textbf{Step 2}, for every point $p\in W$, there exists a closed
      hypersurface $\Gamma(p)\subset W$ passing through $p$ which minimizes
      $\int u_W^{\frac{2k}{k+1}}dA$ in the class of hypersurfaces homologous to $\Gamma$.
    By the rigidity part of Theorem \ref{spectral inheritance}, the following three conclusions
      hold. Let $\nu_p$ be the local unit normal of $\Gamma(p)$ in $W$.
    \begin{enumerate}[label=(\roman*)]
        \item $\Gamma(p)$ is totally geodesic in $W$.
        \item 
        \[
            \frac{2k}{k+1}(\nabla_{W}^{2}\log u_W)(\nu_p,\nu_p)
            =\Ric_{W}(\nu_p,\nu_p)
            \quad\text{on } \Gamma(p).
        \]
        \item $\nabla_Wu_W=0$ along $\Gamma(p)$.
    \end{enumerate}
    Since every point of $W$ lies on some such hypersurface $\Gamma(p)$, it follows that
      $\nabla_Wu_W=0$ at every point of $W$. Thus $u_W$ is constant.
    Then we have
    \[
    \Ric_{W}(\nu_p,\nu_p)=0 \quad \text{on } \Gamma(p).
    \]
    Let $\{e_i\}$ be a local orthonormal basis on $\Gamma(p)$. By step 1, we can construct a sequence of weighted minimizing hypersurfaces converging to
      $\Gamma(p)$. Thus, according to \cite[Lemma 2.11]{chuleezhu_n_systole}, we have 
    \[
    \text{Rm}_{W}(e_{i},\nu_p,e_{j},\nu_p)
      =\text{Rm}_{W}(e_{i},\nu_p,e_{j},e_{\ell})=0.
    \]
    Since $\Gamma(p)$ is totally geodesic in $W$, the Gauss equation identifies the
      sectional curvature of $W$ on two-planes tangent to $\Gamma(p)$ with the
      sectional curvature of $\Gamma(p)$, which is nonnegative by the covering conclusion
      above. The preceding curvature identities show that the mixed sectional curvatures vanish.
      Thus every sectional curvature of $W$ at $p$ is nonnegative. Since $p$ is arbitrary,
      the sectional curvature of $W$ is nonnegative everywhere. In particular,
      $\Ric_{W}\geq0$. Since $u_W$ is constant, the weighted minimizing property of
      $\Gamma$ becomes the usual area-minimizing property in its nonzero homology class.
      Therefore $\Gamma$ is locally area-minimizing in $W$. Hence, by Theorem
      \ref{splitting-under-ric}, $W$ is isometrically covered by
      $\Gamma\times \mathbb{R}$. Since $\Gamma$ is isometrically covered by
      $\mathbb{S}^{n-m+1}\times \mathbb{R}^{m-k-2}$, we obtain that $W$ is isometrically
      covered by $\mathbb{S}^{n-m+1}\times \mathbb{R}^{m-k-1}$.

It remains to prove the case $k=0$. Here $u_{N}\equiv1$ and hypersurface $\Gamma\subset N$
  is an ordinary area-minimizing hypersurface in a nonzero homology class $[N]\frown f^{*}(d\theta_{1})$. Theorem \ref{spectral inheritance} for $k=0$ and  the induction hypothesis implies $\Gamma$ is
  isometrically covered by $\mathbb{S}^{n-m+1}\times\mathbb{R}^{m-2}$. The
  rigidity statement of Theorem \ref{spectral inheritance} for $k=0$ gives that $\Gamma$ is
  totally geodesic in $N$ and that $\Ric_N(\nu,\nu)=0$ along $\Gamma$. We now use the unweighted
  analogue of the preceding surface-capture argument. Since $u_N\equiv1$, the weighted functional
  is the ordinary area functional and all terms involving derivatives of the weight vanish. The
  metric-deformation and open-closed arguments above therefore produce, through every point of
  $N$, an area-minimizing hypersurface homologous to $\Gamma$. Applying the $k=0$ rigidity statement
  of Theorem \ref{spectral inheritance} to these hypersurfaces gives total geodesicity and
  vanishing normal Ricci curvature along each of them. As in Step 3, the one-sided limiting
  hypersurfaces then give the mixed curvature identities, and hence $\Ric_N\geq0$. Since $\Gamma$ is
  area-minimizing in a nonzero homology class, Lemma \ref{splitting-under-ric} implies that $N$ is
  isometrically covered by $\Gamma\times\mathbb{R}$, and hence by
  $\mathbb{S}^{n-m+1}\times\mathbb{R}^{m-1}$.
\end{proof}
    \bibliographystyle{alpha}

\bibliography{reference}

@article{chenhong2026,
      title={Intermediate curvature and splitting theorem}, 
      author={Jingche Chen and Han Hong},
      year={2026},
      eprint={2604.26529},
      archivePrefix={arXiv},
      primaryClass={math.DG},
      url={https://arxiv.org/abs/2604.26529}, 
}

@article{hong-wang-splitting-scalarcurvature,
          title={A splitting theorem for 3-manifold with nonnegative scalar curvature and mean-convex boundary}, 
      author={Han Hong and Gaoming Wang},
     JOURNAL = {arXiv:2501.08677},
  FJOURNAL = {},
    VOLUME = {},
      YEAR = {2025},
    NUMBER = {},
     PAGES = {},
}

@article {chodosh-eichmair-moraru-splitting,
    AUTHOR = {Chodosh, Otis and Eichmair, Michael and Moraru, Vlad},
     TITLE = {A splitting theorem for scalar curvature},
   JOURNAL = {Comm. Pure Appl. Math.},
  FJOURNAL = {Communications on Pure and Applied Mathematics},
    VOLUME = {72},
      YEAR = {2019},
    NUMBER = {6},
     PAGES = {1231--1242},
      ISSN = {0010-3640,1097-0312},
   MRCLASS = {53C21 (53C20)},
  MRNUMBER = {3948556},
MRREVIEWER = {Thomas\ Schick},
       DOI = {10.1002/cpa.21803},
       URL = {https://doi.org/10.1002/cpa.21803},
}

@article {chu-lee-zhu-kaler-splitting,
    AUTHOR = {Chu, Jianchun and Lee, Man-Chun and Zhu, Jintian},
     TITLE = {On {K}\"{a}hler manifolds with non-negative mixed curvature},
   JOURNAL = {J. Reine Angew. Math.},
  FJOURNAL = {Journal f\"{u}r die Reine und Angewandte Mathematik. [Crelle's
              Journal]},
    VOLUME = {827},
      YEAR = {2025},
     PAGES = {313--338},
      ISSN = {0075-4102,1435-5345},
   MRCLASS = {53C55 (53C21)},
  MRNUMBER = {4978422},
       DOI = {10.1515/crelle-2025-0054},
       URL = {https://doi.org/10.1515/crelle-2025-0054},
}

@article {zhu-spherical-systole-dimension-two,
    AUTHOR = {Zhu, Jintian},
     TITLE = {Rigidity of area-minimizing {$2$}-spheres in {$n$}-manifolds
              with positive scalar curvature},
   JOURNAL = {Proc. Amer. Math. Soc.},
  FJOURNAL = {Proceedings of the American Mathematical Society},
    VOLUME = {148},
      YEAR = {2020},
    NUMBER = {8},
     PAGES = {3479--3489},
      ISSN = {0002-9939,1088-6826},
   MRCLASS = {53C24 (53C42)},
  MRNUMBER = {4108854},
MRREVIEWER = {Otis\ Chodosh},
       DOI = {10.1090/proc/15033},
       URL = {https://doi.org/10.1090/proc/15033},
}

@article {bray-brendle-eichmair-neves,
    AUTHOR = {Bray, H. and Brendle, S. and Eichmair, M. and Neves, A.},
     TITLE = {Area-minimizing projective planes in 3-manifolds},
   JOURNAL = {Comm. Pure Appl. Math.},
  FJOURNAL = {Communications on Pure and Applied Mathematics},
    VOLUME = {63},
      YEAR = {2010},
    NUMBER = {9},
     PAGES = {1237--1247},
      ISSN = {0010-3640,1097-0312},
   MRCLASS = {53C42 (53C20)},
  MRNUMBER = {2675487},
MRREVIEWER = {St\'{e}phane\ Sabourau},
       DOI = {10.1002/cpa.20319},
       URL = {https://doi.org/10.1002/cpa.20319},
}

@article {meeks-yau-existence,
    AUTHOR = {Meeks, III, William H. and Yau, Shing Tung},
     TITLE = {Topology of three-dimensional manifolds and the embedding
              problems in minimal surface theory},
   JOURNAL = {Ann. of Math. (2)},
  FJOURNAL = {Annals of Mathematics. Second Series},
    VOLUME = {112},
      YEAR = {1980},
    NUMBER = {3},
     PAGES = {441--484},
      ISSN = {0003-486X},
   MRCLASS = {53C42 (49F10 57M35)},
  MRNUMBER = {595203},
MRREVIEWER = {F.\ J.\ Almgren, Jr.},
       DOI = {10.2307/1971088},
       URL = {https://doi.org/10.2307/1971088},
}

@article {heintze-karcher-submanifold,
    AUTHOR = {Heintze, Ernst and Karcher, Hermann},
     TITLE = {A general comparison theorem with applications to volume
              estimates for submanifolds},
   JOURNAL = {Ann. Sci. \'{E}cole Norm. Sup. (4)},
  FJOURNAL = {Annales Scientifiques de l'\'{E}cole Normale Sup\'{e}rieure.
              Quatri\`eme S\'{e}rie},
    VOLUME = {11},
      YEAR = {1978},
    NUMBER = {4},
     PAGES = {451--470},
      ISSN = {0012-9593},
   MRCLASS = {53C40 (58E10)},
  MRNUMBER = {533065},
MRREVIEWER = {Hubert\ Gollek},
       URL = {http://www.numdam.org/item?id=ASENS_1978_4_11_4_451_0},
}

@article {micallef-moraru-splitting,
    AUTHOR = {Micallef, Mario and Moraru, Vlad},
     TITLE = {Splitting of 3-manifolds and rigidity of area-minimising
              surfaces},
   JOURNAL = {Proc. Amer. Math. Soc.},
  FJOURNAL = {Proceedings of the American Mathematical Society},
    VOLUME = {143},
      YEAR = {2015},
    NUMBER = {7},
     PAGES = {2865--2872},
      ISSN = {0002-9939,1088-6826},
   MRCLASS = {53A10 (49Q05 53C24)},
  MRNUMBER = {3336611},
MRREVIEWER = {Rondinelle\ Marcolino\ Batista},
       DOI = {10.1090/S0002-9939-2015-12137-5},
       URL = {https://doi.org/10.1090/S0002-9939-2015-12137-5},
}

@article {zhou_multiplicity_one_conj,
    AUTHOR = {Zhou, Xin},
     TITLE = {On the multiplicity one conjecture in min-max theory},
   JOURNAL = {Ann. of Math. (2)},
  FJOURNAL = {Annals of Mathematics. Second Series},
    VOLUME = {192},
      YEAR = {2020},
    NUMBER = {3},
     PAGES = {767--820},
      ISSN = {0003-486X,1939-8980},
   MRCLASS = {53C42 (49J35 49Q05 58E12)},
  MRNUMBER = {4172621},
MRREVIEWER = {Martin\ Man-Chun\ Li},
       DOI = {10.4007/annals.2020.192.3.3},
       URL = {https://doi.org/10.4007/annals.2020.192.3.3},
}

@article {toponogov_closed_geodesic,
    AUTHOR = {Toponogov, V. A.},
     TITLE = {Evaluation of the length of a closed geodesic on a convex
              surface},
   JOURNAL = {Dokl. Akad. Nauk SSSR},
  FJOURNAL = {Doklady Akademii Nauk SSSR},
    VOLUME = {124},
      YEAR = {1959},
     PAGES = {282--284},
      ISSN = {0002-3264},
   MRCLASS = {52.00 (53.00)},
  MRNUMBER = {102055},
MRREVIEWER = {P.\ C.\ Hammer},
}

@article {chuleezhu_n_systole,
    AUTHOR = {Chu, Jianchun and Lee, Man-Chun and Zhu, Jintian},
     TITLE = {Homological {$n$}-systole in {$(n + 1)$}-manifolds and
              bi-{R}icci curvature},
   JOURNAL = {Adv. Math.},
  FJOURNAL = {Advances in Mathematics},
    VOLUME = {467},
      YEAR = {2025},
     PAGES = {Paper No. 110187, 26},
      ISSN = {0001-8708,1090-2082},
   MRCLASS = {53C21 (53C24 53C42)},
  MRNUMBER = {4873655},
       DOI = {10.1016/j.aim.2025.110187},
       URL = {https://doi.org/10.1016/j.aim.2025.110187},
}

@article {hong-wang-spectral-splitting,
    AUTHOR = {Han Hong and Gaoming Wang},
     TITLE = {A splitting theorem for manifolds with nonnegative spectral Ricci curvature and mean-convex boundary},
   JOURNAL = {arXiv:2503.07009},
  FJOURNAL = {},
    VOLUME = {},
      YEAR = {2025},
    NUMBER = {},
     PAGES = {},
    
}

@article {antonelli-xu,
      TITLE={New spectral Bishop-Gromov and Bonnet-Myers theorems and applications to isoperimetry}, 
      AUTHOR ={Gioacchino Antonelli and Kai Xu},
      JOURNAL = {	arXiv:2405.08918},
    FJOURNAL = {arXiv:2405.08918},
    VOLUME = {},
      YEAR = {2024},
    NUMBER = {},
     PAGES = {},
      ISSN = {}, 
}

@article {BrayBrenleNevesrigidity,
    AUTHOR = {Bray, Hubert and Brendle, Simon and Neves, Andre},
     TITLE = {Rigidity of area-minimizing two-spheres in three-manifolds},
   JOURNAL = {Comm. Anal. Geom.},
  FJOURNAL = {Communications in Analysis and Geometry},
    VOLUME = {18},
      YEAR = {2010},
    NUMBER = {4},
     PAGES = {821--830},
      ISSN = {1019-8385,1944-9992},
   MRCLASS = {53C24 (53C42)},
  MRNUMBER = {2765731},
MRREVIEWER = {Vincent\ Minerbe},
       DOI = {10.4310/CAG.2010.v18.n4.a6},
       URL = {https://doi.org/10.4310/CAG.2010.v18.n4.a6},
}

@article {chenshuli_end,
    AUTHOR = {Chen, Shuli},
     TITLE = {A generalization of the {G}eroch conjecture with arbitrary
              ends},
   JOURNAL = {Math. Ann.},
  FJOURNAL = {Mathematische Annalen},
    VOLUME = {389},
      YEAR = {2024},
    NUMBER = {1},
     PAGES = {489--513},
      ISSN = {0025-5831,1432-1807},
   MRCLASS = {53C21 (57K50)},
  MRNUMBER = {4735953},
MRREVIEWER = {Harish\ Seshadri},
       DOI = {10.1007/s00208-023-02651-5},
       URL = {https://doi.org/10.1007/s00208-023-02651-5},
}

@article {Schoen-Yau-incompressible-minimal-surfaces,
    AUTHOR = {Schoen, Richard and Yau, Shing Tung},
     TITLE = {Existence of incompressible minimal surfaces and the topology
              of three-dimensional manifolds with nonnegative scalar
              curvature},
   JOURNAL = {Ann. of Math. (2)},
  FJOURNAL = {Annals of Mathematics. Second Series},
    VOLUME = {110},
      YEAR = {1979},
    NUMBER = {1},
     PAGES = {127--142},
      ISSN = {0003-486X},
   MRCLASS = {58E12 (49F10 53C42)},
MRREVIEWER = {Jonathan\ Sacks},
       DOI = {10.2307/1971247},
       URL = {https://doi.org/10.2307/1971247},
}

@article {Carlotto-Chodosh-Eichmair-PMT,
    AUTHOR = {Carlotto, Alessandro and Chodosh, Otis and Eichmair, Michael},
     TITLE = {Effective versions of the positive mass theorem},
   JOURNAL = {Invent. Math.},
  FJOURNAL = {Inventiones Mathematicae},
    VOLUME = {206},
      YEAR = {2016},
    NUMBER = {3},
     PAGES = {975--1016},
      ISSN = {0020-9910,1432-1297},
   MRCLASS = {53C20 (53C24 53C42)},
MRREVIEWER = {Luc\ Nguyen},
       DOI = {10.1007/s00222-016-0667-3},
       URL = {https://doi.org/10.1007/s00222-016-0667-3},
}

@article {Liu-nonnegative-Ricci-curvature,
    AUTHOR = {Liu, Gang},
     TITLE = {3-manifolds with nonnegative {R}icci curvature},
   JOURNAL = {Invent. Math.},
  FJOURNAL = {Inventiones Mathematicae},
    VOLUME = {193},
      YEAR = {2013},
    NUMBER = {2},
     PAGES = {367--375},
      ISSN = {0020-9910,1432-1297},
   MRCLASS = {53C20 (53A10 53C21)},
MRREVIEWER = {David\ J.\ Wraith},
       DOI = {10.1007/s00222-012-0428-x},
       URL = {https://doi.org/10.1007/s00222-012-0428-x},
}

@misc{stern-scalar-curvature-harmonic-maps,
  author = {Stern, Daniel},
  title = {Scalar curvature and harmonic maps to {$S^1$}},
  year = {2019},
  eprint = {1908.09754},
  archivePrefix = {arXiv},
  primaryClass = {math.DG}
}
\end{document}